\documentclass[smallcondensed]{svjour3}  
\smartqed      
\usepackage{graphicx}
\usepackage[framemethod=tikz]{mdframed}

\usepackage{amsmath}
\usepackage{amsfonts}
\usepackage{amssymb}
\usepackage{enumitem}
\usepackage{lscape}
\usepackage{longtable}
\usepackage{rotating}
\usepackage{multirow}
\usepackage{color}
\usepackage{url}
\usepackage{subcaption}
\usepackage{graphicx}
\usepackage{rotating}
\usepackage{fullpage}
\usepackage[ruled,vlined,noline,linesnumbered]{algorithm2e}
\usepackage[colorlinks=true]{hyperref}
\usepackage[normalem]{ulem}
\usepackage{setspace}
\AtBeginDocument{%
  \addtolength\abovedisplayskip{-0.5\baselineskip}%
  \addtolength\belowdisplayskip{-0.5\baselineskip}%
}

\newcommand{\R}{{\mathbb R}}

\DeclareMathOperator{\argmin}{argmin}
\DeclareMathOperator{\prox}{prox}

\newcommand{\cH}{{\mathcal H}}
\newcommand{\cO}{{\mathcal O}}
\newcommand{\cS}{{\mathcal S}}

\newcommand{\cG}{{\mathcal G}}

\newcommand{\demi}{\frac{1}{2}}
\newcommand{\ie}{{\it i.e.}\,\,}

\newcommand{\norm}[1]{\left\|{#1}\right\|}

\newcommand{\tcb}{\textcolor{blue}}

\newcommand{\bpa}[1]{\Big({#1}\Big)}

\def\<{\langle}
\def\>{\rangle}

\if
{
\usepackage[pageref]{backref}
\renewcommand*{\backrefalt}[4]{%
\ifcase #1 %
(Not cited)%
\or
(Cited on p.~#2)%
\else
(Cited on pp.~#2)%
\fi
}

}
\fi

\setlist[enumerate]{label=$(\roman*)$,leftmargin=3pt,align=left}

\begin{document}

\title{Fast convex optimization via closed-loop time scaling of gradient dynamics}
\titlerunning{Accelerated dynamics  for  convex optimization}

\author{Hedy Attouch \and Radu Ioan Bo\c{t} \and Dang-Khoa Nguyen}
\authorrunning{H. Attouch, R.I. Bo\c{t}, D.-K. Nguyen} 

\institute{
Hedy Attouch \at IMAG, Univ. Montpellier, CNRS, Montpellier, France\\
\email{hedy.attouch@umontpellier.fr} \and 
Radu Ioan Bo\c{t}
\at Faculty of Mathematics, University of Vienna, Oskar-Morgenstern-Platz 1, 1090 Vienna, Austria,\\  \email{radu.bot@univie.ac.at}\and 
Dang-Khoa Nguyen
\at Faculty of Mathematics and Computer Science, University of Science, Ho Chi Minh City, Vietnam,
\at Vietnam National University, Ho Chi Minh City, Vietnam, \\\email{ndkhoa@hcmus.edu.vn}
}

\date{\today}

\maketitle

\begin{abstract}
  In a Hilbert setting, for convex differentiable optimization, we develop a general framework for adaptive accelerated gradient methods. They are based on damped inertial dynamics where the coefficients are designed in a closed-loop way. 
 Specifically, the damping is a feedback control of the velocity, or of the gradient of the objective function.
For this, we develop  a closed-loop version of the  time scaling and averaging technique introduced by the authors.
We thus obtain  autonomous inertial dynamics
which involve vanishing viscous damping and  implicit Hessian driven damping.
By simply using the convergence rates for the continuous steepest descent  and Jensen's inequality, without the need for further Lyapunov analysis, we show that the trajectories have several remarkable properties at once: they ensure  fast convergence of  values,  fast convergence of the   gradients towards zero, and they converge to optimal solutions. Our approach leads to parallel algorithmic results,
that we study in the case of proximal algorithms.
These are among the very first general results of this type obtained using autonomous dynamics.  Since the proposed numerical methods are based on proximal techniques, the results can be extended to a broader class, specifically to the problem of minimizing a proper, lower semicontinuous, and convex function. Numerical experiments are conducted to demonstrate the efficiency of the proposed methods.
\end{abstract}

\medskip


\medskip


\bigskip

\section{Introduction}

In a real Hilbert space $\cH$, we develop a dynamic approach to the rapid resolution of convex optimization problems which relies on inertial dynamics whose damping is designed as a closed-loop control.
We consider the  minimization problem

\vspace{-2mm}
\begin{equation}\label{basic-min}
\min \left\lbrace  f(x) : \, x\in\cH  \right\rbrace,
\end{equation}
where, throughout the paper, we make the following  assumptions on the function $f$ to be minimized

\vspace{-2mm}
\begin{equation}\label{eq:basic_hypo}
(\mathcal A) \, \begin{cases}
f: \cH \rightarrow \R  \text{ is a convex function of  class }    {\mathcal C}^1; \, S= \argmin_{\cH} f \neq \emptyset ;
\vspace{1mm}\\
\nabla f   \text{ is Lipschitz continuous on the bounded sets of } \cH.
\end{cases}
\end{equation}
Our study is part of the close links between dissipative dynamical systems and optimization algorithms, the latter being obtained by temporal discretization of the continuous dynamics.
Our study comes as a natural extension of the authors' previous work \cite{ABN} where the technique of time scaling and averaging
was used in an open-loop way, giving rise to non-autonomous damped inertial dynamics with fast convergence properties.
In the present paper, we take advantage of the simplicity and flexibility of  this technique to develop it in a closed-loop way. This will give rise to autonomous damped inertial dynamics with fast convergence properties.
Recall that the low-resolution ODE   obtained by Su, Boyd, and  Cand\`es \cite{SBC} of the accelerated gradient method of Nesterov, together with the corresponding high-resolution ODE \cite{ACFR}, \cite{SDJS} (which involves an additional Hessian driven damping term)  are non-autonomous dynamics, the coefficient of viscous friction being of the form $\alpha/t$. Our study therefore opens a new path in the field of first-order adaptive optimization methods.

\subsection{Time scale and averaging: the open-loop approach}
Let us briefly explain the time scaling and averaging method in the  open-loop case on a model example (see \cite{ABN} for more details). Then we will look at how to develop a corresponding
closed-loop approach.
As the basic starting dynamic, we consider the  continuous steepest descent
\begin{equation}\label{SD}
{\rm (SD)} \quad   \dot{z}(s) + \nabla f (z(s)) =0,
\end{equation}
for which we have the classical convergence result
$$
f \left( z \left( s \right) \right) -\inf\nolimits_{\cH} f = o \left( \dfrac{1}{s} \right) \textrm{ as } s \to + \infty .
$$
Then, we make the change of time variable $s=\tau (t)$ in (SD), where $\tau (\cdot)$ is an increasing  function from $\mathbb R_{+}$ to $\mathbb R_{+}$, continuously differentiable, and  satisfying $\lim_{t \to +\infty}\tau (t) = + \infty$. Setting $y(t):= z(\tau(t))$, we get 

\begin{equation}\label{change var2-b}
 \dot{y} (t)      +   \dot{\tau}(t) \nabla f(y(t)) =0.
\end{equation}
The convergence rate becomes
\begin{equation}\label{SD_rescale_1}
   f (y(t))- \inf\nolimits_{\cH} f = o \left( \dfrac{1}{\tau(t)} \right) \textrm{ as } t \to + \infty.
\end{equation}
Taking $\tau(\cdot)$ which grows faster than the identity, makes the solution trajectories unchanged but travelled faster.
The price to pay is that \eqref{change var2-b} is a non-autonomous dynamic in which the coefficient in front of the gradient term tends to infinity as $t \to +\infty$. 
This prevents from using gradient methods to discretize it. Recall that for gradient methods the step size has to be less than or equal to twice the inverse of the Lipschitz constant of the gradient. To overcome this difficulty we come with the second step of our method which is averaging.
Let us attach to $y(\cdot)$ the new function $x: [t_0, +\infty[ \to \cH$ defined by 

\begin{equation}\label{change var24}
\dot{x}(t) + \frac{1}{\dot{\tau}(t)}(x(t)-y(t))  = 0 ,
\end{equation}
with $x(t_0) =x_0$ given in $\cH$. We shall explain further the averaging interpretation.
Equivalently

\begin{equation}\label{change var240}
y(t)=   x(t)+  \dot{\tau}(t) \dot{x}(t)  .
\end{equation}
 \noindent By temporal derivation of  \eqref{change var240} we get
 
\begin{equation}\label{change var25}
\dot{y}(t) =  \dot{x}(t)+ \ddot{\tau}(t) \dot{x}(t)  + \dot{\tau}(t) \ddot{x}(t)   .
\end{equation}
\noindent Replacing $y(t)$ and $\dot{y}(t)$  as given by \eqref{change var240} and 
\eqref{change var25} in 
\eqref{change var2-b}, we get

\begin{equation}\label{change var27}
\ddot{x}(t) + \dfrac{1+ \ddot{\tau}(t)}{\dot{\tau}(t)}\dot{x}(t) + \nabla f\Big(x(t)+   \dot{\tau}(t) \dot{x}(t) \Big)  = 0 .
\end{equation}
In doing so, we passed from the  first-order  differential equation    \eqref{change var2-b} to the second-order differential equation   \eqref{change var27}, with the advantage that now the coefficient in front of the gradient is fixed.
Let us now particularize the time scale $\tau (\cdot)$. Taking
\begin{equation}\label{change var275_b}
 \tau (t)   = \dfrac{t^2}{2(\alpha -1)} ,
\end{equation}
gives $\frac{1+ \ddot{\tau}(t)}{\dot{\tau}(t)} = \frac{\alpha}{t}$, and the corresponding dynamic with implicit Hessian driven damping

\begin{equation}\label{implicit10}
\ddot{x}(t) + \frac{\alpha}{t} \dot{x}(t) + \nabla f\Big(x(t)+   \dfrac{t}{\alpha -1} \dot{x}(t) \Big)  = 0 .
\end{equation}
In this dynamic, the Hessian driven damping appears in an implicit form. This type of dynamic was initiated in \cite{ALP}, see also \cite{MJ} for a related autonomous system  in the case of a strongly convex function $f$. The rationale justifying the use of the term ``implicit" comes from the observation that by a Taylor expansion (as $t \to +\infty$ we have $t\dot{x}(t) \to 0$ which justifies the use of Taylor expansion), we have

\[
\nabla f\bpa{x(t)+\dfrac{t}{\alpha -1}\dot{x}(t)}\approx \nabla f (x(t)) + \dfrac{t}{\alpha -1}\nabla^2 f(x(t))\dot{x}(t) ,
\]

\noindent thus making the Hessian damping appear indirectly in \eqref{implicit10}.
Because of its important role in attenuating the oscillations,
several recent studies have been devoted to inertial dynamics combining the asymptotic vanishing damping with the geometric Hessian-driven damping (coined sometimes Newton-type inertial dynamics); see e.g., \cite{ABCR,APR2,ACFR,ACFR-Opti,AFK,BCL,CBFP,ML1,SDJS}. In turn, the corresponding algorithms, among which IGAHD enjoys several favorable properties, introduce a correction term in the Nesterov accelerated gradient method (see \cite{Nest1,Nest2}) which reduces the oscillatory aspects. 

 Note that in \eqref{implicit10} the coefficient of the Hessian damping is proportional to the inverse of the viscosity damping. Thus asymptotically when the viscous damping tends towards zero, and therefore can cause many small oscillations to appear, the coefficient of the Hessian driven damping tends towards infinity, and therefore has an effective effect on the attenuation of the oscillations.
This  is the situation considered by Attouch-Bo\c{t}-Nguyen \cite{ABN}, who obtained convergence rates comparable to those associated with the Nesterov accelerated gradient method.
A major advantage of this approach is that there is no need to do a Lyapunov analysis, we only use the classical convergence rate for the continuous steepest descent. Moreover, the convergence of the trajectories is a direct consequence of the known results for the steepest descent.

\subsection{Closed-loop control}

The idea is to exploit the  time scaling and averaging method and the fact that (SD) provides several quantities which are  increasing and converge to $+ \infty$  as $t \to +\infty$, so which are eligible for time scaling. This will enable us to perform  time scaling and averaging in 
a closed-loop way. Indeed, in (SD), the velocity and the norm of the gradient are monotonically decreasing to zero. So, the idea is to use their inverse for defining the time scaling. Specifically, in a first result we are going to define the derivative of the time scaling $\tau(\cdot)$ as a function of the inverse of the speed. This means acceleration of the time scaling  when the speed decreases. Following this approach,  we will obtain in Theorem \ref{thm:basic_inertial_vel}  the following model result.

\begin{theorem}
Suppose that  $f \colon \cH \to \R$ satisfies $(\mathcal A)$.
	Let us choose the positive parameters according  to $q >0$, $p \geq 1$, and $\gamma >1$. Let
	$x \colon \left[ t_{0} , + \infty \right[ \to \cH$ be a solution trajectory of the following system
	
	\begin{equation}
		\label{change var2540vv-intro}
		\begin{cases}
			\ddot{x}(t)+
			\dfrac{(1+\gamma)\dot{\tau}(t)^2 -\tau(t)\ddot{\tau}(t)} {\tau(t)\dot{\tau}(t)} \dot{x}(t)+ \gamma
			\frac{\dot{\tau}(t)^2}{\tau(t)} \nabla f \left(x(t)+  \dfrac{1}{\gamma} \dfrac{\tau(t)} {\dot{\tau}(t)} \dot{x}(t) \right) &=0 \vspace{2mm} \\	
			\tau \left( t \right) - \dfrac{1}{q^{q}} \left( t_{0} + \displaystyle\int_{t_{0}}^{t} \left[ \lambda \left( r \right) \right] ^{\frac{1}{q}} dr \right) ^{q} & = 0 \vspace{2mm} \\
			\left[ \lambda \left( t \right) \right] ^{p} \left\lvert \dot{\tau} \left( t \right) \right\rvert ^{p-1} \left\lVert \nabla f \left( x \left( t \right) + \dfrac{1}{\gamma} \dfrac{\tau(t)} {\dot{\tau}(t)} \dot{x} \left( t \right) \right) \right\rVert ^{p-1} & = 1 .
		\end{cases}
	\end{equation}

	\medskip
	
	\noindent Then  we have the fast convergence of values: as $t\to +\infty$
	
	\begin{equation}\label{SD_rescale_1200bbc}
		f (x(t))- \inf\nolimits_{\cH} f = o \left( \dfrac{1}{t^{1+q-\frac{1}{p}}} \right) .
	\end{equation}
	Moreover, the solution trajectory $x(t)$ converges weakly as $t \to +\infty$, and its limit belongs to $S=\argmin f$.
\end{theorem}

\noindent As a special case, take $p=1$, $q=2$.  Then, the last equation of  \eqref{change var2540vv-intro} gives  $\lambda (t) \equiv 1$. According to this,  the second equation of \eqref{change var2540vv-intro} gives $\tau (t)={\frac{t^2}{4}}$, and we find a case with time scaling in an open-loop form. After elementary calculation, the first equation of \eqref{change var2540vv-intro} is written as

$$
\ddot{x}(t)+
			\dfrac{\alpha} {t} \dot{x}(t)+
			\dfrac{\alpha-1}{2}\nabla f \left( x(t)+ \dfrac{t} {\alpha-1} \dot{x}(t) \right)=0,
			$$
with $\alpha = 2\gamma +1 >3$, and the convergence rate of the values becomes
\begin{equation}\label{SD_rescale_1200bbc8}
		f (x(t))- \inf\nolimits_{\cH} f = o \left( \dfrac{1}{t^{2}} \right) .
	\end{equation}
We therefore recover the results obtained by the authors in the case of the open loop, giving the optimal convergence rates for general convex differentiable optimization.
This inertial formulation may seem at first glance  complicated. Indeed it is equivalent to the first-order system in time and space

\begin{equation}
		\label{defi lambda-vel_b_intro}
		\begin{cases}
			\dot{y} \left( t \right) + \dot{\tau} \left( t \right) \nabla f \left( y \left( t \right) \right) & = 0  \vspace{1mm}\\	
			\dot{x}(t) + \gamma \dfrac{\dot{\tau}(t)}{\tau(t)} x(t)  - \gamma \dfrac{\dot{\tau}(t)}{\tau(t)}y(t) & = 0 \vspace{1mm} \\	
			\tau \left( t \right) - \dfrac{1}{q^{q}} \left( t_{0} + \displaystyle\int_{t_{0}}^{t} \left[ \lambda \left( r \right) \right] ^{\frac{1}{q}} dr \right) ^{q} & = 0 \vspace{1mm} \\
			\left[ \lambda \left( t \right) \right] ^{p} \left\lVert \dot{y} \left( t \right) \right\rVert ^{p-1} & = 1,
		\end{cases}
	\end{equation}
	
\smallskip
	
\noindent whose temporal discretization  provides corresponding optimization algorithms, see  Theorem \ref{thm:prox-conv-31}.

\subsection{Link with the existing literature}

Contrary to the rich literature that has been devoted to non-autonomous damped inertial methods and their links with the fast first-order optimization algorithms  for general convex optimization (in particular the  Nesterov accelerated gradient method), only a small number of papers have been devoted to these questions, based on  autonomous methods. Indeed the heavy ball method of Polyak only provides the asymptotic convergence rate $1/t$ for general convex functions.  The idea is therefore to see if we can mimic the fast convergence properties of the Su, Boyd, and  Cand\`es dynamic  model (see \cite{SBC}) of the Nesterov accelerated gradient method, using autonomous dynamics. A natural idea is to design the damping term, on which is based the optimization properties of the system, in a closed-loop way. In this direction, we can mention the  following  contributions.

\smallskip

a) Our study has a natural link with works devoted to regularized Newton methods  for solving monotone inclusions (and \eqref{basic-min} in particular).
Given a general maximally monotone operator $A : \cH \rightrightarrows \cH$, to overcome the ill-posed character of the  continuous Newton method, in line with \cite{ASv}, Attouch, Redont and Svaiter have  studied in \cite{ARS} the following closed-loop dynamic version of the Levenberg-Marquardt method

\vspace{-1mm}
\begin{eqnarray*}
 \left\{
\begin{array}{rcl}
v(t)  \in     A(x(t)) \hspace{2cm} \\
\rule{0pt}{12pt}
 \|v(t)\|^{\gamma}  \dot{x}(t)  + \beta   \dot{v}(t) + v(t) =0 .
 \end{array}\right.
\end{eqnarray*}

\noindent When $\gamma >1$, they showed the  well-posedness of the above system, and analyzed its convergence properties.
When $A= \nabla f$  this system writes 

$$
\|\nabla f (x(t))\|^{\gamma}  \dot{x}(t)  + \beta   \nabla^2 f (x(t))\dot{x}(t) +\nabla f (x(t))  =0.
$$

\noindent Thus, its inertial version

$$
\ddot{x}(t) + \|\nabla f (x(t))\|^{\gamma}  \dot{x}(t)   + \beta   \nabla^2 f (x(t))\dot{x}(t) +\nabla f (x(t))  =0
$$

\noindent  falls within the framework  of our study with the damping  equal to a closed-loop control of the norm of the gradient.
The techniques developed in \cite{ARS} are particularly useful for studying the well-posedness of  dynamics with implicit features.

\medskip

b)  Although significantly different, our approach has several points in common with the  article by  Lin and Jordan \cite{LJ}.
In this article, the authors study the   closed-loop dynamical system

\begin{equation}
	\label{ds:LJ:fi}
	\begin{cases}
		\dot{y} \left( t \right) + \dot{\tau} \left( t \right) \nabla f \left( x \left( t \right) \right) & = 0 \vspace{1mm}\\
		\dot{x} \left( t \right) + \frac{\dot{\tau} \left( t \right)}{\tau \left( t \right)} \left( x \left( t \right) - y \left( t \right) \right) + \frac{\left[ \dot{\tau} \left( t \right) \right] ^{2}}{\tau \left( t \right)} \nabla f \left( x \left( t\right) \right) & = 0 \vspace{1mm} \\
		\tau \left( t \right) - \dfrac{1}{4} \left( \int_{0}^{t} \sqrt{\lambda \left( r \right)} dr + c \right) ^{2} & = 0 \vspace{1mm} 
		\\
		\left[ \lambda \left( t \right) \right] ^{p} \left\lVert \nabla f \left( x \left(t \right) \right) \right\rVert ^{p-1} & = \theta ,
	\end{cases}
\end{equation}

\noindent where $c > 0$ and $0<\theta< 1$. The corresponding second-order in time damped inertial system writes as follows

\vspace{-2mm}
\begin{equation}
	\label{ds:LJ:fi-inertial}
\ddot{x}(t)+ \dfrac{2\left[ \dot{\tau}(t) \right] ^{2} -\tau(t)\ddot{\tau}(t)} {\tau(t)\dot{\tau}(t)} \dot{x}(t)+
			\dfrac{\left[ \dot{\tau}(t) \right] ^{2}}{\tau(t)} \nabla^2 f \left( x(t) \right)\dot{x}(t) + \frac{\dot{\tau}(t) (\dot{\tau}(t) + \ddot{\tau}(t)  )}{\tau(t)}\nabla f \left( x(t) \right) =0.
\end{equation}
In the above system, the Hessian driven damping comes in an explicit way because of the structure of the first equation which differs from the structure of the continuous steepest descent. In contrast,  in our approach, the first equation is the rescaled continuous steepest descent, and  the Hessian driven damping comes   implicitly.
 Let us highlight some advantages of our approach. 

\smallskip

$\bullet$ Our system is introduced in a natural way by using the time scaling and averaging method. This makes unnecessary to perform a Lyapunov analysis for the inertial system. It has already been done for the continuous steepest descent. This results in a significantly simplified mathematical analysis.

\smallskip

$\bullet$ Our dynamic model contains an additional parameter $q$ which, when $q=2$, gives the setting of Lin and Jordan, and which, when judiciously tuned, gives better convergence rates.

\smallskip

$\bullet$ Our approach provides the weak convergence of the trajectories to optimal solutions.

\smallskip

We shall return later to the precise comparison between the two systems.

\smallskip

c) In \cite{ABC},  Attouch, Bo\c{t} and Csetnek study the convergence properties of 
the Autonomous Damped Inertial Gradient Equation

 \begin{equation*}
\mbox{\rm (ADIGE)}    \qquad \ddot{x}(t) + \mathcal G \Big( \dot{x}(t),  \nabla f({x}(t)), \nabla^2 f({x}(t))\Big)    +  \nabla f (x(t)) = 0,
\end{equation*}
where  the damping term $\mathcal G \Big( \dot{x}(t),  \nabla f (x(t)),  \nabla^2 f({x}(t)) \Big) $ acts as a closed-loop control.
They pay particular attention to the role played by the parameter $r>1$ in the asymptotic convergence analysis of the dynamic 

$$
\ddot x(t) +  \|   \dot x(t)\|^{r-2} \dot x(t)+ \nabla f(x(t))=0.
$$

\noindent They show that the case $r=2$  separates the weak damping ($r>2$) from the strong damping ($r<2$),  which emphasizes the importance of this case. These questions have also been considered by Haraux and Jendoubi 
in \cite{HJ}.

\smallskip

d) In \cite{Song}, Song, Jiang, and  Ma develop an interesting technique for accelerating high-order algorithms under general H\"older continuity assumption.  Their continuous-time framework reduces to an inertial system without
Hessian-driven damping in the first-order setting, which has been proven to be an inaccurate surrogate. 
Although underlying their approach, the acceleration via time scaling, the averaging technique, and the closed-loop tuning of the coefficients are not clearly identified.

\subsection{Organization of the paper}

After a general presentation of the article in the introduction, we provide in Section 
\ref{sec:SD-estimate} a general estimate of the time scaling for the continuous steepest descent when it is defined in a closed-loop way. This is crucial for the rest of the paper.
Then  we specialize these results to  situations of particular interest, and examine in details the case of closed-loop systems induced respectively by  velocities, and then by  gradients. In Section \ref{sec:accelerated}, which is the main part of the paper, we develop the next important step in our approach, which is the averaging operation. This provides  accelerated 
damped inertial dynamics that are autonomous and with fast convergence properties. Finally, in Section \ref{prox} we analyze the fast convergence properties of proximal algorithms which come naturally from the temporal discretization of the continuous dynamics.

\section{Closed-loop time scaling of the steepest descent}\label{sec:SD-estimate}

\subsection{Formulation of the closed-loop time scaling}
Given $t_{0} \geq 0, \, q > 0$, and $p\geq 1$,  the time scale function $\tau \colon \left[ t_{0} , + \infty \right[ \to \R_{++}$ is defined by

\begin{equation}
		\label{defi lambda-vel_b_general}
		\begin{cases}
			\dot{y} \left( t \right) + \dot{\tau} \left( t \right) \nabla f \left( y \left( t \right) \right) & = 0  \vspace{1mm}\\	
			\tau \left( t \right) - \dfrac{1}{q^{q}} \left( t_{0} + \displaystyle\int_{t_{0}}^{t} \left[ \lambda \left( r \right) \right] ^{\frac{1}{q}} dr \right) ^{q} & = 0 \vspace{1mm} \\
			\left[ \lambda \left( t \right) \right] ^{p} \left[ \cG \left( y \left( t \right) \right) \right] ^{p-1} & = 1,
		\end{cases}
	\end{equation}

\noindent where $\cG (\cdot) $ is a given   positive, continuous function  that depends on the information of the trajectory $y(\cdot)$.
This general formalism allows us to unify the various situations coming from different choices of the time scaling as a feedback control of the state of the system.
 For example $\cG$ may be a function of  $y, \dot{y}, f \left( y \right), \nabla f \left( y \right)$ and/or any mixture combination of them.
Then the function $\lambda (\cdot)$ is continuous and it links the coefficient of $\nabla f$, namely $\dot{\tau} (\cdot)$, with the solution trajectory $y(\cdot)$.\\
As a useful result,  note that for every $t \geq t_{0}$, it holds
\begin{eqnarray}
	\dot{\tau} \left( t \right) &=& \dfrac{1}{q^{q-1}} \left( \displaystyle\int_{t_{0}}^{t} \left[ \lambda \left( r \right) \right] ^{\frac{1}{q}} dr + t_{0} \right) ^{q-1} \left[ \lambda \left( t \right) \right] ^{\frac{1}{q}} \nonumber \\
	&=& \left[ \tau \left( t \right) \right] ^{\frac{q-1}{q}} \left[ \lambda \left( t \right) \right] ^{\frac{1}{q}} > 0 . \label{time-scale-1}
\end{eqnarray}

\noindent Moreover, the relations \eqref{defi lambda-vel_b_general} allow us to cover the open-loop case. 
In particular, when $p=1$ it holds $\lambda \left( t \right) = 1$ for every $t \geq t_{0}$. This yields for every $q > 0$
\begin{equation*}
	\tau \left( t\right) = \left( \dfrac{t}{q} \right) ^{q} .
\end{equation*}
Taking further $q := 1$, then $\tau \left( t \right)$ becomes the \emph{regular time} in variable $t$, namely $\tau \left( t \right) = t$ for every $t \geq t_{0}$.\\
Let us  specify the interpretation of \eqref{defi lambda-vel_b_general} as a steepest descent dynamic which is rescaled in time in a closed-loop way.
\begin{proposition}
\label{prop equiv}
Suppose that $f \colon \cH \to \R$ satisfies $(\mathcal A)$.
Let $t_{0} \geq 0,\,  q > 0$, $p \geq 1$ and $y \colon \left[ t_{0} , + \infty \right[ \to \cH$ be a solution trajectory of the  system
\eqref{defi lambda-vel_b_general}. Suppose that
\begin{equation*}
	\lim_{s \to + \infty} \tau \left( s \right) = + \infty .
\end{equation*}
Then $y(\cdot)$ is a solution trajectory of a time rescaled continuous steepest descent {\rm (SD)}, as described below:\\
 Let $s_{0} = \tau \left( t_{0} \right)$ and $z \colon \left[ s_{0} , + \infty \right[ \to \cH$ be a solution trajectory of the following system
\begin{equation}
	\dot{z} \left( s \right) + \nabla f \left( z \left( s \right) \right) = 0 .
\end{equation}
Then we have
\begin{equation*}
	y \left( t \right) = z \left( \tau \left( t \right) \right) \qquad \forall t \geq t_{0} ,
\end{equation*}
and there exists a continuously differentiable function $\sigma \colon \left[ s_{0} , + \infty \right[ \to \R_{++}$ such that
\begin{equation*}
	z \left(s \right) = y \left( \sigma \left( s \right) \right) \qquad \forall s \geq s_{0} .
\end{equation*}
\end{proposition}
\begin{proof}
We already interpreted how to go from a solution trajectory $z(\cdot)$ of (SD) to the closed-loop system above via the time scaling function $\tau(\cdot)$.
Let us now show the reverse direction.
Let $y \colon \left[ t_{0} , + \infty \right[ \to \cH$ be a solution trajectory of \eqref{defi lambda-vel_b_general}.
We have that $\lambda$ is continuous and positive on $\left[ t_{0} , + \infty \right[$, therefore $\tau$ is a monotonically increasing function, hence  injective.
On the other hand, we have $t_{0} = \tau \left( t_{0} \right) = \left( \frac{t_{0}}{q} \right) ^{q}$. Since by assumption $\lim_{t \to + \infty} \tau \left( t \right) = + \infty$, this means $\tau$ is a continuous function whose image contains $\left[ s_{0} , + \infty \right[$, hence surjective.
Combining these premises, we have shown that $\tau$ is a bijection, which means it is invertible.

\noindent Set $\sigma \equiv \tau^{-1}$ and   make the change of time variable $t := \sigma \left( s \right)$ in \eqref{defi lambda-vel_b_general}. Let us define
\begin{equation*}
	z \left( s \right) = y \left( \sigma \left( s \right) \right) = y \left( \tau^{-1} \left( s \right) \right) .
\end{equation*}
Then by the chain rule, we have
\begin{equation*}
	\dot{z} \left( s \right) = \dot{y} \left( \tau^{-1} \left( s \right) \right) \dfrac{1}{\dot{\tau} \left( \tau^{-1} \left( s \right) \right)} = \dot{y} \left( \sigma \left( s \right) \right) \dfrac{1}{\dot{\tau} \left( \sigma \left( s \right) \right)} .
\end{equation*}
This leads to
\begin{equation*}
	\dot{z} \left( s \right) + \nabla f \left( z \left( s \right) \right) = 0 .
\end{equation*}
In other words, $z \colon \left[ s_{0} , + \infty \right[ \to \cH$ is a solution trajectory of (SD).
\qed
\end{proof}

The above assertion allows us to transfer the convergence results of (SD) to some closed-loop systems.
In particular, given a time scaling function $\tau(\cdot)$ as described above, by making the change of time variable $s := \tau \left( t \right)$, we obtain the following results from Theorem \ref{SD_pert_thm} in the appendix applied to the unperturbed continuous steepest descent system.
\begin{eqnarray}
&&\displaystyle{\int_{t_0}^{+\infty} } \dfrac{\tau \left( t \right)}{\dot{\tau} \left( t \right)} \left\lVert \dot{y} \left(t \right)  \right\rVert ^{2} dt  < + \infty , \label{change var vel} \\
&&\displaystyle{\int_{t_0}^{+\infty} } \tau \left( t \right) \dot{\tau} \left( t \right) \left\lVert \nabla f \left( y \left( t \right) \right) \right\rVert ^{2} dt  < + \infty , \label{change var grad} \\
&&f (y(t))- \inf\nolimits_{\cH} f  = o \left( \dfrac{1}{\tau(t)} \right) \textrm{ as } t \to + \infty, \label{change var func}\\
&&  \left\lVert \nabla f \left( y \left( t \right) \right) \right\rVert = o \left( \dfrac{1}{\tau(t)} \right) \textrm{ as } t \to + \infty. \label{change var norm grad}
\end{eqnarray}

\subsection{Lower bound estimate of the time scaling $\tau(t)$}

 As key ingredient of our approach, the next step is to establish a lower bound for $\tau(t)$ in terms of $t$. This will reflect the acceleration of our dynamic via time scaling and allow us to achieve fast  convergence rates. For this, we will need the following technical lemma, which can be seen as a nonlinear Gronwall result.
 
\begin{lemma}\label{lemma tau}
Suppose that there exists $C_{0} > 0$ and $b > a \geq 0$ such that
\begin{equation}
	\int_{t_{0}}^{t} \left[ \tau \left( r \right) \right] ^{a} \left[ \lambda \left( r \right) \right] ^{-b} dr \leq C_{0} < + \infty \quad \forall t \geq t_0.
\end{equation}
Then there exists $C_{1} > 0$ such that
\begin{equation}\label{tau s}
	\tau \left( t \right) \geq C_{1} \left( t - t_{0} \right) ^{\frac{qb+1}{b-a}} \quad \forall t \geq t_0.
\end{equation}
\end{lemma}
\begin{proof}

Let $t \geq t_{0}$ be fixed.  By applying the H\"{o}lder inequality, we get
\begin{align}
	\int_{t_{0}}^{t} \left[ \tau \left( r \right) \right] ^{\frac{a}{qb+1}} dr & \leq \left( \int_{t_{0}}^{t} \left[ \tau \left( r \right) \right] ^{a} \left[ \lambda \left( r \right) \right] ^{-b} dr \right) ^{\frac{1}{qb+1}} \left( \int_{t_{0}}^{t} \left[ \lambda \left( r \right) \right] ^{\frac{1}{q}} dr \right) ^{\frac{qb}{qb+1}} \nonumber \\
	& \leq C_{0}^{\frac{1}{qb+1}} \left( t_{0} + \int_{t_{0}}^{t} \left[ \lambda \left( r \right) \right] ^{\frac{1}{q}} dr \right) ^{\frac{qb}{qb+1}} = \left( C_{0} q^{qb} \right) ^{\frac{1}{qb+1}} \left[ \tau \left( t \right) \right] ^{\frac{b}{qb+1}} . \label{est tau} 
\end{align}

\noindent If $a=0$ then \eqref{tau s} follows immediately.
From now on suppose that $a > 0$, so that the inequality \eqref{est tau} can be rewritten as
\begin{equation}
	\int_{t_{0}}^{t} \left[ \tau \left( r \right) \right] ^{\frac{a}{qb+1}} dr \leq \left( C_{0} q^{qb} \right) ^{\frac{1}{qb+1}} \left( \left[ \tau \left( t \right) \right] ^{\frac{a}{qb+1}} \right) ^{\frac{b}{a}} \label{est Holder}
\end{equation}
The arguments are now adapted from \cite{LJ}, which is inspired by the proof of Bihari-LaSalle inequality.
Let
\begin{equation*}
	C_{q,b} := \left( C_{0} q^{qb} \right) ^{\frac{1}{qb+1}} > 0 
	\qquad \textrm{ and } \qquad
	A \left( t\right) := \int_{t_{0}}^{t} \left[ \tau \left( r \right) \right] ^{\frac{a}{qb+1}} dr \qquad \forall t \geq t_{0} ,
\end{equation*}
so that \eqref{est Holder} becomes
\begin{equation*}
	A \left( t \right) \leq C_{q,b} \left[ \dot{A} \left( t \right) \right] ^{\frac{b}{a}} \quad \forall t \geq t_{0}
\end{equation*}
or, equivalently,
\begin{equation*}
	C_{q,b}^{-\frac{a}{b}} \leq \left[ A \left( t \right) \right] ^{- \frac{a}{b}} \dot{A} \left( t \right) \quad \forall t \geq t_{0}.
\end{equation*}
Integrating from $t_{0}$ to $t$, we obtain
\begin{align*}
	C_{q,b}^{-\frac{a}{b}} \left( t - t_{0} \right) & \leq \left( 1 - \dfrac{a}{b} \right) \left[ \left[ A \left( t \right) \right] ^{1 - \frac{a}{b}} - \left[ A \left( t_{0} \right) \right] ^{1 - \frac{1}{b}} \right] \nonumber \\
	& \leq \left[ A \left( t \right) \right] ^{1 - \frac{a}{b}} 
	\leq \left[C_{q,b} \left[ \tau \left( t \right) \right] ^{\frac{b}{qb+1}} \right] ^{1 - \frac{a}{b}} \nonumber \\
	& = C_{q,b}^{\frac{b-a}{b}} \left[ \tau \left( t \right) \right] ^{\frac{b-a}{qb+1}},
\end{align*}
\smallskip
where the last inequality comes from \eqref{est tau}.
Since $b > a$, the conclusion follows.
\qed
\end{proof}

Let us now particularize our results to some model situations. 
\subsection{Closed-loop control of (SD) via the velocity}

\begin{theorem}\label{SD-CL-velocity}
Suppose that  $f \colon \cH \to \R$ satisfies $(\mathcal A)$.
Let $q > 0$, $p \geq 1$ and $y \colon \left[ t_{0} , + \infty \right[ \to \cH$ be a solution trajectory of the following system
\begin{equation}
	\label{defi lambda-dy}
	\begin{cases}
		\dot{y} \left( t \right) + \dot{\tau} \left( t \right) \nabla f \left( y \left( t \right) \right) & = 0 \vspace{1mm} \\		
		\tau \left(t \right) - \dfrac{1}{q^{q}} \left( t_{0} + \displaystyle\int_{t_{0}}^{t} \left[ \lambda \left( r \right) \right] ^{\frac{1}{q}} dr \right) ^{q} & = 0  \vspace{1mm}\\
		\left[ \lambda \left( t\right) \right] ^{p} \left\lVert \dot{y} \left(  t  \right) \right\rVert ^{p-1} & = 1 .
	\end{cases}
\end{equation}

\smallskip

\noindent Then the following statements are satisfied:
\begin{enumerate}

\item 
(convergence of values) \quad
$
f \left( y \left( t \right) \right) - \inf\nolimits_{\cH} f = o \left( t^{-\left( 1+q-\frac{1}{p} \right)} \right) \textrm{ as } t \to + \infty .
$

\smallskip

\item 
(convergence  of  gradients towards zero)	\quad $\left\lVert \nabla f \left( y \left( t \right) \right) \right\rVert = o \left(t^{-\left( 1+q-\frac{1}{p} \right)} \right) \textrm{ as } t \to + \infty.
$

\item 
(integral estimate of the velocities) \quad
$
\displaystyle{\int_{t_0}^{+\infty}} t^{\left( 1+\frac{1}{q}- \frac{1}{pq} \right)} \left\lVert \dot{y} \left( t \right) \right\rVert ^{2+ \frac{p-1}{pq } } dt < +\infty.
$

\item 
The solution trajectory $y(t)$ converges weakly as $t \to +\infty$, and its limit belongs to $\cS=\argmin f$.
\end{enumerate}
\end{theorem}
\begin{proof}	
When $p = 1$, we recover the open loop case with the time scaling function $\tau(t)=\left( \frac{t}{q} \right) ^{q}$. The result is a direct consequence of Theorem \ref{SD_pert_thm}.
Therefore, from now on we only consider the case $p > 1$.
Recall that from \eqref{change var vel} we have
\begin{equation}\label{est:vel_SD}
	\int_{t_0}^{+\infty} \dfrac{\tau \left( t \right)}{\dot{\tau} \left( t \right)} \left\lVert \dot{y} \left( t\right) \right\rVert ^{2} dt < + \infty .
\end{equation}
By using successively the definition of $\lambda$, and relation \eqref{time-scale-1}, we obtain
\begin{equation*}
	\dfrac{\tau \left( t \right)}{\dot{\tau} \left( t \right)} \left\lVert \dot{y} \left( t \right) \right\rVert ^{2} = \dfrac{\tau \left( t \right)}{\dot{\tau} \left( t \right)} \left[ \lambda \left( t \right) \right] ^{- \frac{2p}{p-1}} = \left[ \tau \left( t \right) \right] ^{\frac{1}{q}} \left[ \lambda \left(t \right) \right] ^{- \frac{1}{q} - \frac{2p}{p-1}} \quad \forall t \geq t_0.
\end{equation*}
According to the two above results we get
\begin{equation*}
	 \int_{t_0}^{+\infty} \left[ \tau \left( r \right) \right] ^{\frac{1}{q}} \left[ \lambda \left( r \right) \right] ^{- \frac{1}{q} - \frac{2p}{p-1}} dr < + \infty .
\end{equation*}
 We are now in position to apply Lemma \ref{lemma tau} with $p>1$, $a := \frac{1}{q}$ and $b := \frac{1}{q} + \frac{2p}{p-1}$. We have
\begin{equation*}
	\dfrac{qb+1}{b-a} = \dfrac{2 + \frac{2pq}{p-1}}{\frac{2p}{p-1}} = \dfrac{p - 1 + pq}{p} = 1 + q - \dfrac{1}{p} ,
\end{equation*}
and therefore there exists some constant $C_{1} > 0$ such that
\begin{equation}\label{basic-growth-lambda}
	\tau \left( t \right) \geq C_{1} \left( t - t_{0} \right) ^{1+q-\frac{1}{p}} \quad \forall t \geq t_{0}.
\end{equation}
This leads to $\lim_{t \to + \infty} \tau \left( t \right) = + \infty$. Therefore, according to Proposition \ref{prop equiv}, we can extract the results from Theorem \ref{SD_pert_thm} and the corresponding formulas \eqref{change var vel}, \eqref{change var grad}, \eqref{change var func}. Specifically, we obtain

$(i)$  for the values
\begin{eqnarray*}
f \left( y \left( t \right) \right) - \inf\nolimits_{\cH} f 
&=& o\left( \frac{1}{\tau (t)}\right)=  o \left(\frac{1}{ t^{1+q-\frac{1}{p}} }\right),
\end{eqnarray*}

$(ii)$ for  the gradients
\begin{eqnarray*}
\left\lVert \nabla f \left( y \left( t \right) \right) \right\rVert  &=& o \left( \dfrac{1}{\tau(t)} \right)=  o \left(\frac{1}{ t^{1+q-\frac{1}{p}} }\right). 
\end{eqnarray*}

$(iii)$ for the velocities: we start from \eqref{est:vel_SD}, \ie
$
	\int_{t_0}^{+\infty} \dfrac{\tau \left( t \right)}{\dot{\tau} \left( t \right)} \left\lVert \dot{y} \left( t\right) \right\rVert ^{2} dt < + \infty,
$
that we evaluate as follows:
\begin{eqnarray*}
\dfrac{\tau \left( t \right)}{\dot{\tau} \left( t \right)} \left\lVert \dot{y} \left( t\right) \right\rVert ^{2}&= &\frac{\tau(t)}{\tau(t)^{\frac{q-1}{q}} \lambda(t)^{\frac{1}{q}}}\left\lVert \dot{y} \left( t\right) \right\rVert ^{2}\\
&= &\frac{\tau(t)^{\frac{1}{q}}}{ \lambda(t)^{\frac{1}{q}}}\left\lVert \dot{y} \left( t\right) \right\rVert ^{2}\\
&= &\tau(t)^{\frac{1}{q}}
\left\lVert \dot{y} \left( t\right) \right\rVert ^{2+\frac{p-1}{pq}} \quad \forall t \geq t_0.
\end{eqnarray*}

According to \eqref{basic-growth-lambda} we deduce that 
$$
\displaystyle{\int_{t_0}^{+\infty}} t^{\left( 1+\frac{1}{q}- \frac{1}{pq} \right)} \left\lVert \dot{y} \left( t \right) \right\rVert ^{2+ \frac{p-1}{pq } } dt < +\infty.
$$

 $(iv)$ Let us finally examine the convergence of the solution trajectories. We know that the solution trajectory of the continuous steepest descent converges weakly when $t \to +\infty$, and its limit belong to $S= \argmin_{\cH} f \neq \emptyset$; see Theorem \ref{SD_pert_thm} in appendix. With our notation we therefore have that $z(s)$ converges weakly when $s \to +\infty$. Since $\tau(t) \to +\infty$ as $t \to +\infty$, we immediately deduce that  $y(t) =z (\tau(t)$ converges  weakly as $t \to +\infty$, and its limit belong to $S= \argmin_{\cH} f \neq \emptyset$. This completes the proof. \qed
\end{proof}

\subsection{Closed-loop control of (SD) via the norm of gradient}

We develop an analysis parallel to that of the previous section, replacing speed control with gradient control.

\begin{theorem}\label{SD-CL-gradient}
	Suppose that  $f \colon \cH \to \R$ satisfies $(\mathcal A)$.
	Let $q \geq \frac{1}{2}$, $p \geq 1$ and $y \colon \left[ t_{0} , + \infty \right[ \to \cH$ be a solution trajectory of the following system
	
	\begin{equation}
		\label{defi lambda-grad}
		\begin{cases}
			\dot{y} \left( t \right) + \dot{\tau} \left( t \right) \nabla f \left( y \left( t \right) \right) & = 0 \vspace{1mm}\\		
			\tau \left( t \right) - \dfrac{1}{q^{q}} \left( t_{0} + \displaystyle\int_{t_{0}}^{t} \left[ \lambda \left( r \right) \right] ^{\frac{1}{q}} dr \right) ^{q} & = 0 \vspace{1mm}\\
			\left[ \lambda \left( t \right) \right] ^{p} \left\lVert \nabla f \left( y \left( t \right) \right) \right\rVert ^{p-1} & = 1 .
		\end{cases}
	\end{equation}
	
\noindent 	Then the following statements are satisfied:
	\begin{enumerate}
	
	\item 
		(convergence of values) \quad
		$
		f \left( y \left( t \right) \right) - \inf\nolimits_{\cH} f = o \left( t^{-pq} \right) \textrm{ as } t \to + \infty .
		$
		
		\smallskip
		
		\item 
		(convergence  of  gradients towards zero)	\quad $\left\lVert \nabla f \left( y \left( t \right) \right) \right\rVert = o \left( t^{-pq} \right) \textrm{ as } t \to + \infty.
		$
		
		\item 
		(integral estimate of the gradients) \quad 
		$
		\displaystyle{\int_{t_0}^{+\infty}} t^{pq \left( 2-\frac{1}{q} \right)} \left\lVert \nabla f \left( y \left( t \right) \right) \right\rVert ^{2+ \frac{p-1}{pq } } dt < +\infty	.
		$

		\item 
		The solution trajectory $y(t)$ converges weakly as $t \to +\infty$, and its limit belongs to $\cS=\argmin f$.
	\end{enumerate}
\end{theorem}
\begin{proof}
Again, we only consider the case $p > 1$.  We know from \eqref{change var grad} that
\begin{equation*}
	\displaystyle{\int_{t_0}^{+\infty} } \tau \left( t \right) \dot{\tau} \left( t \right) \left\lVert \nabla f \left( y \left( t \right) \right) \right\rVert ^{2} dt < +\infty.
\end{equation*}
By using successively the definition of $\lambda$, and the relation \eqref{time-scale-1}, we obtain
\begin{equation*}
	\tau \left( t \right) \dot{\tau} \left( t \right) \left\lVert \nabla f \left( y \left( t \right) \right) \right\rVert ^{2} = \tau \left( t \right) \dot{\tau} \left( t \right) \left[ \lambda \left( t\right) \right] ^{- \frac{2p}{p-1}} = \left[ \tau \left( t \right) \right] ^{2 - \frac{1}{q}} \left[ \lambda \left( t\right) \right] ^{\frac{1}{q} - \frac{2p}{p-1}} \quad \forall t \geq t_0.
\end{equation*}
Therefore
\begin{equation*}
	\int_{t_0}^{+\infty} \left[ \tau \left( t \right) \right] ^{2 - \frac{1}{q}} \left[ \lambda \left( t\right) \right] ^{\frac{1}{q} - \frac{2p}{p-1}}  dt < + \infty .
\end{equation*}

\noindent Let us apply Lemma \ref{lemma tau} with
 $a := 2 - \frac{1}{q}$ and  $b = \frac{2p}{p-1} - \frac{1}{q}$. We have $b>a$,  $a\geq 0$ for $q\geq \demi$, and 
\begin{equation*}
	\dfrac{qb+1}{b-a} = \dfrac{\frac{2pq}{p-1}}{\frac{2p}{p-1}-2} = pq .
\end{equation*}
Therefore
\begin{equation}\label{eq:est-tau-grad}
	\tau \left( t \right) \geq C_{1} \left( t - t_{0} \right) ^{pq} \quad \forall t \geq t_0.
\end{equation}

\noindent This gives $\lim_{t \to + \infty} \tau \left( t \right) = + \infty$. According to Proposition \ref{prop equiv}, we can extract the results from Theorem \ref{SD_pert_thm} and the corresponding formulas \eqref{change var vel}, \eqref{change var grad}, \eqref{change var func}. Specifically, we obtain

\smallskip

$(i)$  for the values
\begin{eqnarray*}
f \left( y \left( t \right) \right) - \inf\nolimits_{\cH} f 
&=& o\left( \frac{1}{\tau (t)}\right)=  o \left(\frac{1}{ t^{pq} }\right);
\end{eqnarray*}

$(ii)$   for  the gradients
\begin{eqnarray*}
\left\lVert \nabla f \left( y \left( t \right) \right) \right\rVert  &=& o \left( \dfrac{1}{\tau(t)} \right)=  o \left( t^{-pq} \right);
\end{eqnarray*}

$(iii)$  for the integral estimate of the gradients: we start from \eqref{est:vel_SD}
$$
	\int_{t_0}^{+\infty}  \tau \left( t \right) \dot{\tau} \left( t \right) \left\lVert \nabla f \left( y \left( t \right) \right) \right\rVert ^{2} dt < + \infty,
$$
that we evaluate as follows:
\begin{eqnarray*}
\tau \left( t \right) \dot{\tau} \left( t \right) \left\lVert \nabla f \left( y \left( t \right) \right) \right\rVert ^{2}&= & \tau \left( t \right) \left[ \tau \left( t \right) \right] ^{\frac{q-1}{q}} \left[ \lambda \left( t \right) \right] ^{\frac{1}{q}}  \left\lVert \nabla f \left( y \left( t \right) \right) \right\rVert ^{2}\\
&= &\tau(t)^{2-\frac{1}{q}} \lambda(t)^{\frac{1}{q}}\left\lVert \nabla f \left( y \left( t \right) \right) \right\rVert ^{2}\\
&= &\tau(t)^{2-\frac{1}{q}}
\left\lVert \nabla f (y(t)) \right\rVert ^{2-\frac{p-1}{pq}} \quad \forall t \geq t_0.
\end{eqnarray*}
According to \eqref{eq:est-tau-grad} we deduce that 
$$
\displaystyle{\int_{t_0}^{+\infty}} t^{pq \left( 2-\frac{1}{q} \right)} \left\lVert \nabla f \left( y \left( t \right) \right) \right\rVert ^{2+ \frac{p-1}{pq } } dt < +\infty.
$$

$(iv)$ The convergence of the solution trajectory follows from an  argument similar to that of the previous section.
This completes the proof. \qed
\end{proof}

\begin{remark}{\rm
a) We thus achieved our first goal which was to accelerate the convergence properties of the  continuous steepest descent using closed-loop time scaling.
For example, concerning the convergence rate of  the values, we passed from the convergence rate $1/t$ for the steepest descent to  
$1/t^{\left( 1+q-\frac{1}{p} \right)}$ when the closed-loop control acts on the velocity, and  $1/t^{pq}$ in the case of the gradient.
Clearly, by playing with the parameters $p$ and $q$ we can get arbitrary fast convergence results.
The same observation holds for the convergence of the gradients towards zero.

b) By introducing a time scale function $\tau(\cdot)$ which grows faster than the identity (\ie $\tau(t) \geq t$) either in  open-loop or closed-loop, we  have thus accelerated the continuous steepest descent dynamic.
The price to pay is that we no longer have an autonomous dynamic in \eqref{change var2-b}, with as major drawback the fact that the coefficient in front of the gradient term tends towards infinity as $t \to +\infty$. 
This prevents from using gradient methods to discretize it. Recall that for gradient methods, the step size has to be less than or equal to twice the inverse of the Lipschitz constant of the gradient. To overcome this, we come with the second step of our method which is averaging.}
\end{remark}

\section{Accelerated gradient systems with closed-loop control of the damping}\label{sec:accelerated}

\subsection{General results concerning time scale and averaging}

We will prove the following general result which puts forward a damped inertial dynamics which comes by time scale and averaging of the  continuous steepest descent. Then we will specialize it and consider time scale obtained in a closed-loop way, and thus  cover the two model situations.

\begin{theorem}\label{thm:basic_inertial_general}
	Suppose that  $f \colon \cH \to \R$ satisfies $(\mathcal A)$.
	Let $\gamma >1$, and let $\tau : [t_0,+\infty[ \to \mathbb R_{++}$ be an increasing  function, continuously differentiable, such that $\lim_{t \to +\infty}\tau (t) = + \infty$. Let
	$x \colon \left[ t_{0} , + \infty \right[ \to \cH$ be a solution trajectory of the following second-order differential equation
	
	\begin{equation}
		\label{change var2540v}
		\ddot{x}(t)+
		\frac{(1+\gamma)\dot{\tau}(t)^2 -\tau(t)\ddot{\tau}(t)} {\tau(t)\dot{\tau}(t)} \dot{x}(t)+ \gamma
		\frac{\dot{\tau}(t)^2}{\tau(t)} \nabla f \left(x(t)+  \dfrac{1}{\gamma} \frac{\tau(t)} {\dot{\tau}(t)} \dot{x}(t) \right)=0.
	\end{equation}
 Then we have the convergence rate of the values: as $t\to +\infty$	
	\begin{equation}\label{basic101}
		f\left( x(t)\right)  -\inf\nolimits_{\cH} f = o \left( 
			\dfrac{1}{\tau \left( t \right)} \right), 
	\end{equation}
	
	\smallskip
	
	\noindent and  $x(t)$ converges weakly as $t \to +\infty$, and its limit belongs to $S=\argmin f$.
	
	\smallskip
	
\end{theorem}

\begin{proof}
	a)  We first prove that the trajectory $x(\cdot)$ can be seen to be obtained from the trajectory of the continuous steepest descent via time scale and averaging.  We start from  $y(\cdot)$ solution of 
	
	\begin{equation}\label{SD_rescale_basic_001}
		\dot{y} \left( t \right) + \dot{\tau} \left( t \right) \nabla f \left( y \left( t \right) \right)  = 0. 
	\end{equation}
	According to the time scale analysis developed in \eqref{SD_rescale_1} we have
	
	\begin{equation*}
		f \left( y \left( t \right) \right) - \inf\nolimits_{\cH} f = o \left( \dfrac{1}{\tau \left( t \right)} \right) \textrm{ as } t \to + \infty .
	\end{equation*}
	
\noindent 	This means there exists a positive function $\varepsilon$ which satisfies $\lim_{t \to + \infty} \varepsilon \left( t \right) = 0$ and
	\begin{equation}\label{SD_rescale_120}
		f \left( y \left( t \right) \right) - \inf\nolimits_{\cH} f = \dfrac{\varepsilon \left( t \right)}{\tau \left( t \right)} \quad \forall t \geq t_0.
	\end{equation}

	\noindent
	Let us define the time averaging process as the transformation from $y$ to $x$ according to the formula
	
	\begin{equation}\label{change var241}
		\dot{x}(t) + \gamma \frac{\dot{\tau}(t)}{\tau(t)} x(t)  = \gamma \frac{\dot{\tau}(t)}{\tau(t)}y(t),
	\end{equation}
where $\gamma > 1$.
	Equivalently
	\begin{equation}\label{change var251}
		y(t)= x(t)+ \dfrac{1}{\gamma} \frac{\tau(t)} {\dot{\tau}(t)} \dot{x}(t).
	\end{equation}
	By derivating $y(\cdot)$ we get

	\begin{equation}\label{change var251-derivative}
		\dot{y} \left( t \right)= \dot{x} \left( t \right)+ \dfrac{1}{\gamma} \frac{\tau(t)} {\dot{\tau}(t)} \ddot{x}(t) + \dfrac{1}{\gamma}
		\frac{\dot{\tau}(t)^2 -\tau(t)\ddot{\tau}(t)} {\dot{\tau}(t)^2} \dot{x}(t).
	\end{equation}
	Replacing $\dot{y} \left( t \right)$ by this expression in  the constitutive rescaled steepest descent equation
	\eqref{SD_rescale_basic_001},
	we get
	\begin{equation*}
		\dot{x} \left( t \right) + \dfrac{1}{\gamma} \frac{\tau(t)} {\dot{\tau}(t)} \ddot{x}(t)+ \dfrac{1}{\gamma}
		\frac{\dot{\tau}(t)^2 -\tau(t)\ddot{\tau}(t)} {\dot{\tau}(t)^2} \dot{x}(t)+
		\dot{\tau} \left( t \right) \nabla f \left( x(t)+ \dfrac{1}{\gamma} \frac{\tau(t)} {\dot{\tau}(t)} \dot{x}(t) \right)=0.
	\end{equation*}
	Equivalently
	\begin{equation*}
		\dfrac{1}{\gamma} \frac{\tau(t)} {\dot{\tau}(t)} \ddot{x}(t)+
		\frac{(1+\gamma)\dot{\tau}(t)^2 -\tau(t)\ddot{\tau}(t)} {\gamma \dot{\tau}(t)^2} \dot{x}(t)+ 
		\dot{\tau} \left( t \right) \nabla f \left(x(t)+ \dfrac{1}{\gamma}  \frac{\tau(t)} {\dot{\tau}(t)} \dot{x}(t) \right)=0.
	\end{equation*}
	After multiplication by $\gamma\frac{\dot{\tau}(t)}{\tau(t)} $ we get
	\begin{equation}\label{change var254}
		\ddot{x}(t)+
		\frac{(1+\gamma)\dot{\tau}(t)^2 -\tau(t)\ddot{\tau}(t)} {\tau(t)\dot{\tau}(t)} \dot{x}(t)+ \gamma
		\frac{\dot{\tau}(t)^2}{\tau(t)} \nabla f \left(x(t)+  \dfrac{1}{\gamma} \frac{\tau(t)} {\dot{\tau}(t)} \dot{x}(t) \right)=0.
	\end{equation}
	
	\smallskip
	
	\noindent b) Let us now come to the corresponding estimate of the convergence rates  with $x(t)$ instead of $y(t)$.
	The idea is to express $x$ as an average of $y$, and then conclude thanks to Jensen's inequality.  
	Set 
	
	\begin{eqnarray}
		b(t)&=& \frac{\dot{\tau}(t)}{\tau(t)} \geq 0 \label{change var249b} \\
		B(t)&= & \int_{t_0}^t b(u)du 
		= \int_{t_0}^t \frac{\dot{\tau}(u)}{\tau(u)} du = 
		\ln \left( \frac{\tau (t)}{\tau(t_0)}\right) \label{change var249B}.
	\end{eqnarray}
	Therefore
	\begin{equation}\label{change var249}
		e^{B(t) }= \frac{\tau (t)}{\tau(t_0)}.
	\end{equation} 
	\noindent In order to express $x$ in terms of $y$, we need to integrate the first-order linear differential equation \eqref{change var241} which is written equivalently as follows
	
	\begin{equation*}
		\dot{x}(t) + \gamma b(t)x(t)= \gamma b(t)y(t) .
	\end{equation*}
	After multiplying by $e^{\gamma B(t) }$, we get equivalently
	
	\begin{equation*}
		e^{\gamma B(t) }\dot{x}(t) + \gamma b(t) e^{\gamma B(t) } x(t)= \gamma b(t)e^{\gamma B(t) }y(t) ,
	\end{equation*}
	that is, 
	
	\begin{equation*}
		\dfrac{d}{dt} \left( e^{\gamma B(t)}x(t) \right)= \gamma b(t)e^{\gamma B(t)}y(t) .
	\end{equation*}
	After integration we get
	
	\begin{equation*}
		e^{\gamma B(t)}x(t) =  e^{\gamma B(t_0)}x(t_0)  +    \gamma  \int_{t_0}^t  b(u)e^{\gamma B(u)}y(u) du .
	\end{equation*}
	According to $e^{\gamma B(t_0)} =e^0 =1$ we get
	\begin{eqnarray}
		x(t) &=&  e^{-\gamma B(t)} x(t_0)  +   \gamma e^{-\gamma B(t)}  \int_{t_0}^t  b(u)e^{\gamma B(u)}y(u))du \nonumber \\
		&=&   e^{-\gamma B(t)}y(t_0)  + \gamma e^{-\gamma B(t)}  \int_{t_0}^t  b(u)e^{\gamma B(u)}y(u)du, \label{change var247}
	\end{eqnarray}
	where the last equality follows from the choice of the Cauchy data $y(t_0)= x(t_0)$.
	Then, observe that 
	$x(t)$ can be simply written as follows
	\begin{equation}\label{proba-formulation}
		x(t) =   \int_{t_0}^t y(u)\,  d\mu_{t} (u), 
	\end{equation}
	where $\mu_t$ is the positive  Radon  measure on $[t_0,t]$ defined by
	\begin{equation}\label{proba-formulation2}
		\mu_t = e^{-\gamma B(t)} \delta_{t_0} + \gamma b(u)e^{\gamma \left( B(u)-B(t) \right)} du.
	\end{equation}
	Precisely, in \eqref{proba-formulation2}, $\delta_{t_0}$ is the Dirac measure at $t_0$, and $b(u)e^{B(u)-B(t)} du$ is the measure with density $ b(u)e^{B(u)-B(t)}$ with respect to the Lebesgue measure on $[t_0,t]$.
	According to
	$$
	\gamma e^{-\gamma B(t)} \int_{t_0}^t b(u)e^{\gamma B(u)}du= 1-  e^{-\gamma B(t)},
	$$
	we have  that $\mu_t$ is a positive Radon measure on $[t_0, t]$  whose total mass is equal to $1$. It is therefore a probability measure, and $x(t)$ is obtained by \textbf{averaging} the trajectory $y(\cdot)$ on $[t_0,t]$ with respect to  $\mu_t$.
	From there, let us show how to deduce fast convergence properties for the so defined trajectory $x(\cdot)$.
	According to the convexity of $f$, and  \textbf{Jensen's inequality}, we deduce that
	\begin{eqnarray*}
		f\left(\int_{t_0}^t y(u)\,  d\mu_{t} (u)\right)  -\inf\nolimits_{\cH} f &=& (f -\inf\nolimits_{\cH} f ) \left( \int_{t_0}^t y(u)  d\mu_t (u)\right)\\
		&\leq& \int_{t_0}^t  \left( f (y(u)) -\inf\nolimits_{\cH} f \right) d\mu_t (u)
		\\
		&=&  \int_{t_0}^t  \dfrac{\varepsilon \left( u \right)}{\tau \left( u \right)} d\mu_t (u),
	\end{eqnarray*}
	where the last  inequality above comes from \eqref{SD_rescale_120}.
	According to the definition of $\mu_t$ (see \eqref{proba-formulation2}) and the formulation of $x(t)$ (see \eqref{proba-formulation}), we deduce that 
	\begin{eqnarray*}
		f\left( x(t)\right)  -\inf\nolimits_{\cH} f 
		&\leq& \dfrac{\varepsilon \left( t_{0} \right)}{\tau \left( t_{0} \right)} e^{-\gamma B(t)}   +   \gamma e^{-\gamma B(t)}  \int_{t_0}^t  \dfrac{\varepsilon \left( u \right)}{\tau \left( u \right)}  b(u)e^{\gamma B(u)}du .
	\end{eqnarray*}
	Equivalently,
	\begin{equation}\label{basic01}
		\tau \left( t \right) \left( f\left( x(t)\right)  -\inf\nolimits_{\cH} f  \right)
		\leq \varepsilon \left( t_{0} \right) \left( \dfrac{\tau \left( t \right)}{\tau \left( t_{0} \right)} \right) ^{1-\gamma}    +    \gamma \tau \left( t \right) e^{-\gamma B(t)}  \int_{t_0}^t  \dfrac{\varepsilon \left( u \right)}{\tau \left( u \right)}  b(u)e^{\gamma B(u)}du .
	\end{equation}
	Since $\gamma > 1$ and $\lim_{t \to +\infty}\tau (t) = + \infty$, it holds
	\begin{equation*}
		\limsup_{t \to + \infty} \tau \left( t \right) \left( f\left( x(t)\right)  -\inf\nolimits_{\cH} f  \right)
		\leq \gamma \limsup_{t \to + \infty} \tau \left( t \right) e^{-\gamma B(t)}  \int_{t_0}^t  \dfrac{\varepsilon \left( u \right)}{\tau \left( u \right)}  b(u)e^{\gamma B(u)}du .
	\end{equation*}
	It is therefore enough to show that
	\begin{equation*}
		\limsup_{t \to + \infty} \left( \gamma \tau \left( t \right) e^{-\gamma B(t)}  \int_{t_0}^t  \dfrac{\varepsilon \left( u \right)}{\tau \left( u \right)}  b(u)e^{\gamma B(u)}du \right) \leq 0 .
	\end{equation*}

	\noindent  In order to prepare for integration by parts, note that
	\begin{equation*}
		\gamma b(u)e^{\gamma B(u)}= \frac{d}{du} \left( e^{\gamma B(u)} \right) \textrm{ and } \dfrac{\dot{\tau} \left( u \right)}{\left[ \tau \left( u \right) \right] ^{2-\gamma}} = \dfrac{d}{du} \left( \dfrac{1}{\gamma - 1} \dfrac{1}{\left[ \tau \left( t \right) \right] ^{1- \gamma}} \right) .
	\end{equation*}
	Given an arbitrary $\eta > 0$ we consider $T_{\eta} > t_{0}$ such that $\varepsilon \left( u \right) \leq \eta$ for every $u \geq T_{\eta}$. Therefore, for every $t \geq T_{\eta}$, by integration by parts  and by taking into consideration the relations \eqref{change var249b}-\eqref{change var249}, we get
	\begin{align*}
		& \gamma \tau \left( t \right) e^{-\gamma B(t)}  \int_{t_0}^t  \dfrac{\varepsilon \left( u \right)}{\tau \left( u \right)} b(u)e^{\gamma B(u)}du \nonumber \\
		= \ & \gamma \tau \left( t \right) e^{-\gamma B(t)} \left(  \int_{t_0}^{T_{\eta}} \dfrac{\varepsilon \left( u \right)}{\tau \left( u \right)} b(u)e^{\gamma B(u)}du + \int_{T_{\eta}}^t  \dfrac{\varepsilon \left( u \right)}{\tau \left( u \right)} b(u)e^{\gamma B(u)}du \right) \nonumber \\
		\leq \ & \tau \left( t \right) e^{-\gamma B(t)} \left( \gamma \int_{t_0}^{T_{\eta}}  \dfrac{\varepsilon \left( u \right)}{\tau \left( u \right)} b(u)e^{\gamma B(u)}du  + \eta \gamma \int_{T_{\eta}}^t  \dfrac{1}{\tau \left( u \right)} b(u)e^{\gamma B(u)}du \right) \nonumber \\
		= \ & \tau \left( t \right) e^{-\gamma B(t)} \left( \gamma \int_{t_0}^{T_{\eta}}  \dfrac{\varepsilon \left( u \right)}{\tau \left( u \right)} b(u)e^{\gamma B(u)}du  + \dfrac{\eta}{\tau \left( t \right)} e^{\gamma B(t)} - \dfrac{\eta}{\tau \left( T_{\eta} \right)} e^{\gamma B(T_{\eta})}
		+ \eta \int_{T_{\eta}}^t \dfrac{\dot{\tau} \left( u \right)}{\left[ \tau \left( u \right) \right] ^{2}}  e^{\gamma B(u)}du  \right) \nonumber \\
		= \ & \tau \left( t \right) e^{-\gamma B(t)} \left( \gamma \int_{t_0}^{T_{\eta}}  \dfrac{\varepsilon \left( u \right)}{\tau \left( u \right)} b(u)e^{\gamma B(u)}du  + \dfrac{\eta}{\tau \left( t \right)} e^{\gamma B(t)} - \dfrac{\eta}{\tau \left( T_{\eta} \right)} e^{\gamma B(T_{\eta})}
		+ \dfrac{\eta}{\left[ \tau \left( t_0 \right) \right] ^{\gamma}} \int_{T_\eta}^t \dfrac{\dot{\tau} \left( u \right)}{\left[ \tau \left( u \right) \right] ^{2-\gamma}}  du  \right) \nonumber \\
		= \ & \tau \left( t \right) e^{-\gamma B(t)} \left( \gamma \int_{t_0}^{T_{\eta}}  \dfrac{\varepsilon \left( u \right)}{\tau \left( u \right)} b(u)e^{\gamma B(u)}du  + \dfrac{\eta}{\tau \left( t \right)} e^{\gamma B(t)} - \dfrac{\eta}{\tau \left( T_{\eta} \right)} e^{\gamma B(T_{\eta})}
		+ \dfrac{\eta \left[ \tau \left( t_{0} \right) \right] ^{-\gamma}}{\gamma - 1} \left( \dfrac{1}{\left[ \tau \left( t \right) \right] ^{1- \gamma}} - \dfrac{1}{\left[ \tau \left(T_\eta \right) \right] ^{1- \gamma}} \right) \right) \nonumber \\
		\leq \ & \left( \gamma \int_{t_0}^{T_{\eta}}  \dfrac{\varepsilon \left( u \right)}{\tau \left( u \right)} b(u)e^{\gamma B(u)}du \right) \tau \left( t \right) e^{-\gamma B(t)} + \eta + \dfrac{\eta}{\gamma - 1} \nonumber \\
		\leq  \ 	& C \left[ \tau \left( t \right) \right] ^{1 - \gamma} + \dfrac{\eta \gamma}{\gamma - 1} .
	\end{align*}
Since $\lim_{t \to +\infty}\tau (t) = + \infty$, and $\gamma >1$,
	we obtain 
	\begin{equation}\label{basic1}
		\limsup_{t \to + \infty} \left( \gamma \tau \left( t \right) e^{-\gamma B(t)}  \int_{t_0}^t  \dfrac{\varepsilon \left( u \right)}{\tau \left( u \right)}  b(u)e^{\gamma B(u)}du \right) \leq \dfrac{\eta \gamma}{\gamma - 1} . 
	\end{equation}
This being  true for every $\eta > 0$, we infer
	\begin{equation}\label{basic101_b}
		f\left( x(t)\right)  -\inf\nolimits_{\cH} f  = o \left( 
		\dfrac{1}{\tau \left( t \right)} \right). 
	\end{equation}

	\smallskip
	
	c) For trajectories convergence, we take advantage of the fact that the solution trajectory $ z \left(\cdot \right)$ of the continuous steepest descent  converges weakly towards a solution $x_{*} \in S$. Since  $\lim_{t \to +\infty}\tau (t) = + \infty$, this immediately implies that $y(t) = z(\tau(t))$ converges weakly to $x_{*}$ as $s \to +\infty$. In other words, for each $v \in \cH$
	
	$$
	\left\langle y \left( t \right) , v \right\rangle \to \left\langle x_{*} , v \right\rangle \textrm{ as } t \to + \infty .
	$$
	To pass from the convergence of $y$ to that of $x$, we use the interpretation of $x$ as an average of $y$. The convergence then results from the general property which says that convergence entails ergodic convergence. Let us make this precise.
	Using again that $\lim_{t \to +\infty}\tau (t) = + \infty$, we have
	\begin{equation*}
		x(t) \sim \gamma e^{-\gamma B \left( t \right)} \int_{t_{0}}^{t} b \left( u \right) e^{\gamma B \left( u \right)} y \left( u \right) du = \frac{\gamma}{\left[ \tau (t) \right] ^{\gamma}} \int_{t_{0}}^t \dot{\tau} \left( u \right) \left[ \tau (u) \right] ^{\gamma-1} y(u)du.
	\end{equation*}
	
	\noindent After elementary calculus, we just need to prove that if 
	$a (\cdot)$ is a positive real-valued function which verifies $\lim_{u \to + \infty} a(u)=0$, then $\lim_{t \to + \infty} A(t)=0$, where
	$$
	A(t) =   \frac{\gamma}{\left[ \tau (t) \right] ^{\gamma}} \int_{t_{0}}^t \dot{\tau} \left( u \right) \left[ \tau (u) \right] ^{\gamma-1} a(u)du.
	$$
	Given an arbitrary $\eta >0$, let us take $T_{\eta}$ such that $t_{0} <T_{\eta} $ and $
	a(u) \leq \eta$ for $u \geq T_{\eta}$.
	For $t > T_{\eta}$, we have
	\begin{eqnarray*}
		A(t) &=&   \frac{\gamma}{\left[ \tau (t) \right] ^{\gamma}} \int_{t_{0}}^{T_{\eta}} \dot{\tau} \left( u \right) \left[ \tau (u) \right] ^{\gamma-1} a(u)du + \frac{\gamma}{\left[ \tau (t) \right] ^{\gamma}} \int_{T_{\eta}}^t \dot{\tau} \left( u \right) \left[ \tau (u) \right] ^{\gamma-1} a(u)du\\
		&\leq & \frac{\gamma}{\left[ \tau (t) \right] ^{\gamma}} \int_{t_{0}}^{T_{\eta}} \dot{\tau} \left( u \right) \left[ \tau (u) \right] ^{\gamma-1} a(u)du + \eta \tau (t_0).
	\end{eqnarray*}

\noindent 	Letting $t$ converge to $+\infty$ we get
	
	\begin{equation*}
		\limsup_{t\to +\infty } A(t) \leq \eta \tau (t_0).
	\end{equation*}
	This being true for any $\eta >0$, we infer that $\lim_{t \to + \infty} A(t)=0$, which completes the proof.
	\qed
\end{proof}

\begin{remark}
	By taking $\gamma := \frac{\alpha-1}{2}$ and  $\tau \left( t \right) := \frac{t^{2}}{2 \left( \alpha - 1 \right)}$, equation \eqref{change var254} becomes (see \cite{ABN})
	\begin{equation*}
		\ddot{x}(t)+
		\dfrac{\alpha}{t} \dot{x}(t)+\nabla f \left(x(t)+  \frac{t}{\alpha - 1} \dot{x}(t) \right)=0 .
	\end{equation*}
	
\noindent We have $\gamma > 1 \Longleftrightarrow \alpha > 3$, which is in accordance with the convergence results attached to  Nesterov method.
\end{remark}

\subsection{Damped inertial system via closed-loop control of the velocity}

Let us now examine the model situation where the time scaling is defined in a closed-loop way as a feedback control of the velocity. Completing this construction with the averaging process, as described as above, we get that $(x,y) \colon \left[ t_{0} , + \infty \right[ \to \cH \times \cH$ is a solution trajectory of the following algebraic-differential system

\begin{equation}
	\label{defi lambda-vel_bbb}
	\begin{cases}
		\dot{y} \left( t \right) + \dot{\tau} \left( t \right) \nabla f \left( y \left( t \right) \right) & = 0  \vspace{2mm}\\	
		\dot{x}(t) + \gamma \dfrac{\dot{\tau}(t)}{\tau(t)} x(t) - \gamma \dfrac{\dot{\tau}(t)}{\tau(t)}y(t)& = 0 \vspace{2mm} \\	
		\tau \left( t \right) - \dfrac{1}{q^{q}} \left( t_{0} + \displaystyle\int_{t_{0}}^{t} \left[ \lambda \left( r \right) \right] ^{\frac{1}{q}} dr \right) ^{q} & = 0 \vspace{2mm} \\
		\left[ \lambda \left( t \right) \right] ^{p} \left\lVert \dot{y} \left( t \right) \right\rVert ^{p-1} & = 1 .
	\end{cases}
\end{equation} 

\smallskip

\noindent By specializing Theorem \ref{thm:basic_inertial_general} to this situation we get the following result.

\begin{theorem}\label{thm:basic_inertial_vel}
	Suppose that  $f \colon \cH \to \R$ satisfies $(\mathcal A)$.
	Let $q >0$, $p \geq 1$, $\gamma >1$ and 
	$x \colon \left[ t_{0} , + \infty \right[ \to \cH$ be a solution trajectory of the following system
	
	\begin{equation}
		\label{change var2540vv}
		\begin{cases}
			\ddot{x}(t)+
			\dfrac{(1+\gamma)\dot{\tau}(t)^2 -\tau(t)\ddot{\tau}(t)} {\tau(t)\dot{\tau}(t)} \dot{x}(t)+ \gamma
			\frac{\dot{\tau}(t)^2}{\tau(t)} \nabla f \left(x(t)+  \dfrac{1}{\gamma} \dfrac{\tau(t)} {\dot{\tau}(t)} \dot{x}(t) \right) &=0 \vspace{2mm} \\	
			\tau \left( t \right) - \dfrac{1}{q^{q}} \left( t_{0} + \displaystyle\int_{t_{0}}^{t} \left[ \lambda \left( r \right) \right] ^{\frac{1}{q}} dr \right) ^{q} & = 0 \vspace{2mm} \\
			\left[ \lambda \left( t \right) \right] ^{p} \left\lvert \dot{\tau} \left( t \right) \right\rvert ^{p-1} \left\lVert \nabla f \left( x \left( t \right) + \dfrac{1}{\gamma} \dfrac{\tau(t)} {\dot{\tau}(t)} \dot{x} \left( t \right) \right) \right\rVert ^{p-1} & = 1 .
		\end{cases}
	\end{equation}

	\medskip
	
	\noindent Then  we have the fast convergence of values: as $t\to +\infty$
	
	\begin{equation}\label{SD_rescale_1200bbc2}
		f (x(t))- \inf\nolimits_{\cH} f = o \left( \dfrac{1}{t^{1+q-\frac{1}{p}}} \right) .
	\end{equation}
	Moreover, the solution trajectory $x(t)$ converges weakly as $t \to +\infty$, and its limit belongs to $S=\argmin f$.
	
\end{theorem}

\begin{proof}
	We showed in the proof of Theorem \ref{thm:basic_inertial_general}  how to pass from \eqref{defi lambda-vel_bbb} to \eqref{change var2540vv}. Conversely, let  $x(\cdot)$ be  a solution trajectory of the damped inertial dynamic \eqref{change var2540vv}. Let us   show that by setting
	
	\begin{equation*}
		y \left( t \right) = \dfrac{1}{\gamma} \dfrac{\tau \left( t \right)}{\dot{\tau} \left( t \right)} \dot{x} \left( t \right) + x \left( t \right) ,
	\end{equation*}
	
	\noindent	then $(x,y) \colon \left[ t_{0} , + \infty \right[ \to \cH \times \cH$ is a solution trajectory of
	
	\begin{equation}
		\label{defi lambda-vel_b}
		\begin{cases}
			\dot{y} \left( t \right) + \dot{\tau} \left( t \right) \nabla f \left( y \left( t \right) \right) & = 0  \vspace{2mm}\\	
			\dot{x}(t) + \gamma \dfrac{\dot{\tau}(t)}{\tau(t)} x(t) - \gamma \dfrac{\dot{\tau}(t)}{\tau(t)}y(t)& = 0 \vspace{2mm} \\	
			\tau \left( t \right) - \dfrac{1}{q^{q}} \left( t_{0} + \displaystyle\int_{t_{0}}^{t} \left[ \lambda \left( r \right) \right] ^{\frac{1}{q}} dr \right) ^{q} & = 0 \vspace{2mm} \\
			\left[ \lambda \left( t \right) \right] ^{p} \left\lVert \dot{y} \left( t \right) \right\rVert ^{p-1} & = 1 .
		\end{cases}
	\end{equation}
	
	\medskip
	
	\noindent	Indeed, by taking the time derivative of $y(\cdot)$, as given by the second equation of \eqref{defi lambda-vel_b}, we get
	\begin{align*}
		\dot{y} \left( t \right) & = \dfrac{1}{\gamma} \dfrac{\tau \left( t \right)}{\dot{\tau} \left( t \right)} \ddot{x} \left( t \right) + \dfrac{1}{\gamma} \left( 1+\gamma - \dfrac{\tau \left( t \right) \ddot{\tau} \left( t \right)}{\left[ \dot{\tau} \left( t \right) \right] ^{2}} \right) \dot{x} \left( t \right) \nonumber \\
		& = \dfrac{1}{\gamma} \dfrac{\tau \left( t \right)}{\dot{\tau} \left( t \right)} \left( \ddot{x} \left( t \right) + \dfrac{\left( 1+\gamma \right) \left[ \dot{\tau}(t) \right] ^{2} -\tau(t)\ddot{\tau}(t)} {\tau(t)\dot{\tau}(t)} \dot{x}(t) \right) \nonumber \\
		& = - \dot{\tau} \left( t \right) \nabla f \left( x \left( t \right) + \dfrac{1}{\gamma} \dfrac{\tau \left( t \right)}{\dot{\tau} \left( t \right)} \dot{x} \left( t \right) \right) = - \dot{\tau} \left( t \right) \nabla f \left( y \left( t \right) \right) .
	\end{align*}
	
	\noindent 	This gives  the first equation in \eqref{defi lambda-vel_b} and 
	$$
	\left[ \lambda \left( t \right) \right] ^{p} \left\lVert \dot{y} \left( t \right) \right\rVert ^{p-1}=	\left[ \lambda \left( t \right) \right] ^{p} \left\lvert \dot{\tau} \left( t \right) \right\rvert ^{p-1} \left\lVert \nabla f \left( x \left( t \right) + \dfrac{1}{\gamma} \dfrac{\tau \left( t \right)}{\dot{\tau} \left( t \right)} \dot{x} \left( t \right) \right) \right\rVert ^{p-1}  = 1 .
	$$
	This shows the equivalence of the two systems.
	According to Theorem \ref{SD-CL-velocity}, and formula \eqref{basic-growth-lambda},  there exists a constant $C_{1} > 0$ such that
	
	\begin{equation}
		\tau \left( t \right) \geq C_{1} \left( t - t_{0} \right) ^{1+q-\frac{1}{p}} .
	\end{equation}
	
	\noindent Therefore  $\lim_{t \to +\infty}\tau (t) = + \infty$. According to Theorem  \ref{thm:basic_inertial_general}   we deduce 
	\begin{equation}\label{SD_rescale_1200bb2}
		f\left( x(t)\right)  -\inf\nolimits_{\cH} f = o \left( \dfrac{1}{t^{1+q-\frac{1}{p}}} \right),  
	\end{equation}
	and the convergence of the trajectory.
	\qed	
\end{proof}

\subsection{Damped inertial system via closed-loop control of the gradient}

We proceed in parallel to the previous section to obtain the following result.

\begin{theorem}\label{thm:inertial_grad}
	Suppose that  $f \colon \cH \to \R$ satisfies $(\mathcal A)$.
	Let $q >0$, $p \geq 1$, $\gamma >1$,  and 
	$x \colon \left[ t_{0} , + \infty \right[ \to \cH$ be a solution trajectory of the following system 
	
	\begin{equation}
		\label{change var2540}
		\begin{cases}
			\ddot{x}(t)+
			\dfrac{(1+\gamma)\dot{\tau}(t)^2 -\tau(t)\ddot{\tau}(t)} {\tau(t)\dot{\tau}(t)} \dot{x}(t)+ \gamma
			\frac{\dot{\tau}(t)^2}{\tau(t)} \nabla f \left(x(t)+  \dfrac{1}{\gamma} \frac{\tau(t)} {\dot{\tau}(t)} \dot{x}(t) \right)&=0 \vspace{2mm} \\	
			\tau \left( t \right) - \dfrac{1}{q^{q}} \left( t_{0} + \displaystyle\int_{t_{0}}^{t} \left[ \lambda \left( r \right) \right] ^{\frac{1}{q}} dr \right) ^{q} & = 0 \vspace{2mm} \\
			\left[ \lambda \left( t \right) \right] ^{p} \left\lVert \nabla f \left( x \left( t \right) + \dfrac{1}{\gamma} \dfrac{\tau \left( t \right)}{\dot{\tau} \left( t \right)} \dot{x} \left( t \right) \right) \right\rVert ^{p-1} & = 1 .
		\end{cases}
	\end{equation}

	\smallskip
	
	\noindent Then  we have the fast convergence of values: as $t\to +\infty$
	
	\begin{equation}\label{SD_rescale_1200bb10}
		f (x(t))- \inf\nolimits_{\cH} f  = o \left( \dfrac{1}{t^{pq}} \right).
	\end{equation}
	Moreover, the solution trajectory $x(t)$ converges weakly as $t \to +\infty$, and its limit belongs to $S=\argmin f$.
\end{theorem}
\begin{proof}
	Let $x(\cdot)$ be a solution trajectory of the damped inertial dynamic \eqref{change var2540}. Let us show show that by setting
	\begin{equation*}
		y \left( t \right) = \dfrac{1}{\gamma}  \dfrac{\tau \left( t \right)}{\dot{\tau} \left( t \right)} \dot{x} \left( t \right) + x \left( t \right) ,
	\end{equation*}
	then $(x,y) \colon \left[ t_{0} , + \infty \right[ \to \cH \times \cH$ is a solution trajectory of
	
	\begin{equation}
		\label{defi lambda-grad_b}
		\begin{cases}
			\dot{y} \left( t \right) + \dot{\tau} \left( t \right) \nabla f \left( y \left( t \right) \right) & = 0  \vspace{2mm}\\	
			\dot{x}(t) + \gamma \dfrac{\dot{\tau}(t)}{\tau(t)} x(t) - \gamma \dfrac{\dot{\tau}(t)}{\tau(t)}y(t)& = 0 \vspace{2mm} \\	
			\tau \left( t \right) - \dfrac{1}{q^{q}} \left( t_{0} + \displaystyle\int_{t_{0}}^{t} \left[ \lambda \left( r \right) \right] ^{\frac{1}{q}} dr \right) ^{q} & = 0 \vspace{2mm} \\
			\left[ \lambda \left( t \right) \right] ^{p} \left\lVert \nabla f \left( y \left( t \right) \right) \right\rVert ^{p-1} & = 1 .
		\end{cases}
	\end{equation}
	
	\smallskip
	
	\noindent 	Indeed, by the same argument as for the velocity case, we get
\begin{equation*}
			\dot{y} \left( t \right) = - \dot{\tau} \left( t \right) \nabla f \left( y \left( t \right) \right) .
	\end{equation*}
	
	\noindent 	This gives  the first equation in \eqref{defi lambda-grad_b} and 
	$$
	\left[ \lambda \left( t \right) \right] ^{p} \left\lVert \nabla f({y} \left( t \right)) \right\rVert ^{p-1}=	\left[ \lambda \left( t \right) \right] ^{p}  \left\lVert \nabla f \left( x \left( t \right) + \dfrac{1}{\gamma} \dfrac{\tau \left( t \right)}{\dot{\tau} \left( t \right)} \dot{x} \left( t \right) \right) \right\rVert ^{p-1}  = 1 .
	$$
	This shows the equivalence of the two systems.
	According to Theorem \ref{SD-CL-gradient}, and formula \eqref{eq:est-tau-grad},  there exists a constant $C_{1} > 0$ such that
	
	\begin{equation}
		\tau \left( t \right) \geq C_{1} \left( t - t_{0} \right) ^{pq}.
	\end{equation}
	
	\noindent Therefore from Theorem  \ref{thm:basic_inertial_general}   we deduce, as $t \to +\infty$ 
	\begin{equation}\label{SD_rescale_1200bb22}
		f\left( x(t)\right)  -\inf\nolimits_{\cH} f = o \left( \dfrac{1}{t^{pq}} \right),  
	\end{equation}
	and the convergence of the trajectory.
	\qed	
\end{proof}

\subsection{Comparison with the Lin-Jordan approach}

In \cite{LJ}, the authors study the second-order closed-loop dynamical system

\begin{equation}
	\label{ds:LJ}
	\begin{cases}
		\ddot{x} \left( t \right) + \left( \dfrac{2 \dot{\tau} \left( t \right)}{\tau \left( t \right)} - \dfrac{\ddot{\tau} \left( t \right)}{\dot{\tau} \left( t \right)} \right) \dot{x} \left( t \right) + \dfrac{\left[ \dot{\tau} \left( t \right) \right] ^{2}}{\tau \left( t \right)} \nabla^{2} f \left( x \left( t \right) \right) \dot{x} \left( t \right) + \dfrac{\dot{\tau} \left( t \right) \left[ \dot{\tau} \left( t \right) + \ddot{\tau} \left(t \right) \right]}{\tau \left( t \right)} \nabla f \left( x \left( t \right) \right) & = 0 \vspace{2mm} \\
		\tau \left( t\right) - \dfrac{1}{4} \left( \int_{0}^{t} \sqrt{\lambda \left( t \right)} dr + c \right) ^{2} & = 0  \vspace{2mm}
		\\
		\left[ \lambda \left( t\right) \right] ^{p} \left\lVert \nabla f \left( x \left( t \right) \right) \right\rVert ^{p-1} & = \theta,
	\end{cases}
\end{equation}

\medskip

\noindent whose  first-order reformulation reads

\begin{equation}
	\label{ds:LJ:fi_bb}
	\begin{cases}
		\dot{y} \left( t \right) + \dot{\tau} \left( t \right) \nabla f \left( x \left( t \right) \right) & = 0 \vspace{1mm}\\
		\dot{x} \left( t \right) + \frac{\dot{\tau} \left( t \right)}{\tau \left( t \right)} \left( x \left( t \right) - y \left( t \right) \right) + \frac{\left[ \dot{\tau} \left( t \right) \right] ^{2}}{\tau \left( t \right)} \nabla f \left( x \left( t\right) \right) & = 0 \vspace{1mm} \\
		\tau \left( t \right) - \dfrac{1}{4} \left( \int_{0}^{t} \sqrt{\lambda \left( t \right)} dr + c \right) ^{2} & = 0 \vspace{1mm} 
		\\
		\left[ \lambda \left( t \right) \right] ^{p} \left\lVert \nabla f \left( x \left(t \right) \right) \right\rVert ^{p-1} & = \theta ,
	\end{cases}
\end{equation}

\noindent where $c > 0$ and $0 < \theta < 1$ are given parameters.  See also \cite{LJ-COT} and \cite{LJ-MOR} for some extensions to monotone equations and monotone inclusions, respectively.

\smallskip

a) In \cite{LJ}, the authors obtained the following convergence rate of function values 

\begin{equation*}
	f \left( x \left( t \right) \right) - \inf\nolimits_{\cH} f = \cO \left(
	\frac{1}{ t^{\frac{3p+1}{2}}}
	\right) \textrm{ as } t \to + \infty.
\end{equation*}

\noindent Note that the last two equations in \eqref{ds:LJ:fi_bb} are nothing else than those in \eqref{defi lambda-grad} with $q:=2$. \\
For comparison, in our approach the  convergence rate of the values obtained in Theorem \ref{thm:inertial_grad} when $q=2$   is 

\begin{equation*}
	f \left( x \left( t \right) \right) - \inf\nolimits_{\cH} f = o \left( \frac{1}{t^{2p}} \right)
\end{equation*}

\noindent  which is better for every $p > 1$.

b) Let us now compare the convergence estimates of the gradients. 
In \cite{LJ}, the authors obtain the  integral estimate

\begin{equation*}
	\displaystyle{\int_{t_0}^{+\infty}} t^{\frac{3p+1}{2}} \left\lVert \nabla f \left( x \left( t \right) \right) \right\rVert ^{\frac{p+1}{p}}  dt < +\infty,
\end{equation*}
which leads to

\begin{equation*}
	\inf_{t_{0} \leq \sigma \leq t} \left\lVert \nabla f \left( x \left( \sigma \right) \right) \right\rVert = \cO \left( t^{-\frac{3p}{2}} \right) \textrm{ as } t \to + \infty.
\end{equation*}
In our approach, the right variable to consider is $y(t)$, instead of $x(t)$. According to \eqref{change var grad}  we have 

$$
\displaystyle{\int_{t_0}^{+\infty} } \tau \left( t \right) \dot{\tau} \left( t \right) \left\lVert \nabla f \left( y \left( t \right) \right) \right\rVert ^{2} dt  < + \infty.
$$

\noindent Since $q=2$, according to \eqref{time-scale-1} we have 
$$
\dot{\tau} \left( t \right)= \left[ \tau \left( t \right) \right] ^{\frac{1}{2}} \left[ \lambda \left( t \right) \right] ^{\frac{1}{2}} .
$$
Therefore

\begin{eqnarray*}
	\tau \left( t \right) \dot{\tau} \left( t \right) \left\lVert \nabla f \left( y \left( t \right) \right) \right\rVert ^{2} &=&
	\tau \left( t \right)^{\frac{3}{2}}  \left[ \lambda \left( t \right) \right] ^{\frac{1}{2}} \left\lVert \nabla f \left( y \left( t \right) \right) \right\rVert ^{2} =
	\tau \left( t \right)^{\frac{3}{2}}  \left\lVert \nabla f \left( y \left( t \right) \right) \right\rVert ^{2- \frac{p-1}{2p}} .
\end{eqnarray*}

\noindent Since $\tau(t) \geq C t^{2p}$, we deduce that 

\begin{equation*}
	\int_{t_{0}}^{+ \infty} t^{3p} \left\lVert \nabla f \left( y \left( t \right) \right) \right\rVert ^{\frac{3p+1}{2p}} dt < + \infty .
\end{equation*}

\noindent which leads to 

\begin{equation*}
	\inf_{t_{0} \leq \sigma \leq t} \left\lVert \nabla f \left( y \left( \sigma \right) \right) \right\rVert = \cO \left( t^{-2p} \right) \textrm{ as } t \to + \infty.
\end{equation*}

\noindent Again, our approach gives a better convergence rate than \cite{LJ}.
Let us also specify that our analysis provides the convergence of the trajectories, which is an open question for \cite{LJ}.
Moreover, since our approach is consistent with the steepest continuous descent, it can  naturally be extended to the non-smooth case, and to the case of cocoercive operators, as it was done in the open-loop case in \cite{ABN}.

\subsection{The limiting case $\gamma =1$}
Our previous results are valid under the assumption $\gamma >1$. It is a natural question to examine the limiting case $\gamma =1$. Close examination of the proof of the theorem reveals a slight change in the integration procedure and a logarithm factor appears. The corresponding result obtained is written as follows.

 \begin{theorem}\label{thm:basic_inertial_general_c}
Suppose that  $f \colon \cH \to \R$ satisfies $(\mathcal A)$.
	Let
$x \colon \left[ t_{0} , + \infty \right[ \to \cH$ be a solution trajectory of the following second-order differential equation

	\begin{equation}
		\label{change var2540v_c}
			\ddot{x}(t)+
			\dfrac{2\left[ \dot{\tau}(t) \right] ^{2} -\tau(t)\ddot{\tau}(t)} {\tau(t)\dot{\tau}(t)} \dot{x}(t)+
			\dfrac{\left[ \dot{\tau}(t) \right] ^{2}}{\tau(t)} \nabla f \left( x(t)+ \dfrac{\tau(t)} {\dot{\tau}(t)} \dot{x}(t) \right) =0 
	\end{equation}
where $\tau : \left[ t_{0} , + \infty \right[ \to \mathbb R_{++}$ is an increasing  function, continuously differentiable, and  satisfying $\lim_{t \to +\infty}\tau (t) = + \infty$. Then we have the convergence rate of the values: as $t\to +\infty$	
	\begin{equation}\label{basic101_c}
f\left( x(t)\right)  -\inf\nolimits_{\cH} f  = o \left( 
 \dfrac{ \ln \left( \tau \left( t \right) \right)}{\tau \left( t \right)} \right), 
\end{equation}

\smallskip

\noindent and the solution trajectory $x(t)$ converges weakly as $t \to +\infty$, and its limits belongs to $S=\argmin f$.

\smallskip
	
\noindent Suppose moreover that there exists some $\theta >0$ and $C_1 >0$ such that for $t$ sufficiently large

\begin{equation}
	\label{defi gen closed-loop-bbbb_c}
(\mathcal A)_{\rm asymp} \qquad     \tau(t) \geq C_1 \left( t - t_{0} \right) ^{\theta}.
\end{equation}

\medskip

\noindent Then  we have the fast convergence of values: as $t\to +\infty$

\begin{equation}\label{SD_rescale_1200bb_c}
   f\left( x(t)\right)  -\inf\nolimits_{\cH} f = o \left( 
 \dfrac{ \ln \left(t \right)}{t^{\theta}} \right).  
\end{equation}
\end{theorem}
When specialized to the closed-loop control of the velocity, we obtain
\begin{equation}\label{SD_rescale_1200bb2_c}
   f\left( x(t)\right)  -\inf\nolimits_{\cH} f = o \left( 
 \dfrac{ \ln \left(t \right)}{t^{1+q-\frac{1}{p}}} \right),  
\end{equation}
and in the case of the  closed-loop control of the gradient  

\begin{equation}\label{SD_rescale_1200bb2_d}
   f\left( x(t)\right)  -\inf\nolimits_{\cH} f = o \left( 
 \dfrac{ \ln \left(t \right)}{t^{pq}} \right).  
\end{equation} 
So, the convergence rates are in this limiting case a little worse because of the logarithm term.

\section{Associated proximal algorithms}\label{prox}

\subsection{A proximal-explicit discretization}

In the following, we present a numerical approach based on a \emph{proximal-explicit} temporal discretization of the closed-loop systems investigated in this paper. By proximal-explicit we mean that the function $f$ is evaluated using a proximal step while the step size sequence $\left( \lambda_{k} \right) _{k \geq 0}$ and the time scaling sequence $\left( \tau_{k} \right) _{k \geq 0}$ are computed explicitly.  This makes our numerical scheme much easier implementable than the numerical algorithm proposed in \cite{LJ} as well as the \textsf{large-step A HPE} approach by Monteiro and Svaiter \cite{MS} which are in fact approximations of a \emph{proximal-implicit} discrete time method.
 We restrict ourselves to the case $q=1$, which gives $\dot{\tau} \left( t \right) = \lambda \left( t \right)$. In this case,  the continuous time closed-loop dynamical system is written as follows

\begin{equation}
	\label{defi gen closed-loop-num}
	\begin{cases}
		\dot{y} \left( t \right) + \lambda \left( t \right) \nabla f \left( y \left( t \right) \right) & = 0 \vspace{2mm}\\		
		\left[ \lambda \left( t \right) \right] ^{p} \left[ \cG \left( y \left( t \right) \right) \right] ^{p-1} & = 1.
	\end{cases}
\end{equation}

\noindent
Let us describe the general structure of the algorithm which is obtained by a proximal-explicit discretization of the continuous system \eqref{defi gen closed-loop-num}.

Given $y_k, y_{k-1}$ in $\cH$, we first define $\lambda_k$ by
$$
\left[ \lambda_k \right] ^{p} \left[ \cG \left( y_k, y_{k-1} \right) \right] ^{p-1}  = 1 .
$$
and consider then an implicit finite difference scheme for the first equation of  \eqref{defi gen closed-loop-num}
\begin{equation}
	\label{im 1e}
	y_{k+1} - y_{k} + \lambda_{k} \nabla f \left( y_{k+1} \right) = 0 .
\end{equation}

\noindent This gives the following algorithm, called (PEAS) for Proximal Explicit Algorithm with Adaptive Step Size.

\begin{algorithm}[h]	
	\caption{Proximal-explicit algorithm with adaptive step size (PEAS)}
	\label{algo:prox-explicit-genral}
	\vspace{2mm}
	\DontPrintSemicolon 
	
	\KwIn{$y_{0} \neq y_{-1} \in \cH$}
	\For{$k = 0, 1, \cdots$}
	{	
			$\lambda_{k} := \left[ \cG \left( y_k, y_{k-1} \right) \right]^{-\frac{p-1}{p}}$\\
			$y_{k+1}	 := \prox_{\lambda_{k} f} \left( y_{k} \right)$
}	
\end{algorithm}

\noindent Note  that  $\left( \lambda_{k} \right) _{k \geq 0}$ is computed explicitly in terms of $\left( y_{k} \right) _{k \geq 0}$.  In other words, the definition of the sequence $\left( \lambda_{k} \right) _{k \geq 0}$ is decoupled from the computation of $\left( y_{k} \right) _{k \geq 0}$. 
This is different from the method in \cite{LJ}, which ultimately leads to the \textsf{large-step A HPE} approach by Monteiro and Svaiter in \cite{MS}.

Let us now specify the link between $\lambda_k $ and $\tau_{k}$.
We start from the relation (recall that we take $q=1$)

\begin{equation}\label{def:tau-lambda}
	\dot{\tau} \left( t \right) = \lambda \left( t \right) .
\end{equation}

\noindent  Then, for every $k \geq 0$ we discretize \eqref{def:tau-lambda} as follows

\begin{equation}
	\label{im dtau}
	\tau_{k+1} - \tau_{k} = \lambda_k \Longleftrightarrow \tau_{k+1} = \lambda_{k} + \tau_{k}
\end{equation}

\noindent  with the convention $\lambda_{0} :=  t_{0}$ and $\tau_0 :=0$, which then yields\;
$\tau_{k} = \sum_{i = 0}^{k-1} \lambda_{i} .$

\noindent Drawing inspiration from continuous analysis, we will first show that the function value $f \left( y_{k} \right) - \inf\nolimits_{\cH} f$ attains the $o \left( \frac{1}{\tau_{k+1}} \right)$ rate of convergence, and the sequence $\left( y_{k} \right) _{k \geq 0}$ converges weakly to a solution. Then, as a crucial result, we will  derive a lower bound of $\tau_{k+1}$ in terms of $k$.

\noindent The following result emphasizes that the rate of convergence and summability results holds for $\left( y_{k} \right) _{k \geq 0}$ for arbitrary step sizes  $\lambda_{k}$ that satisfy $\sum_{k \geq 0} \lambda_{k} =+\infty$.  The proof is an adaptation of  \cite[Theorem 4.1]{ABN}.
\begin{theorem}\label{thm:prox1}
Let $y_0 \in \cH$, $\left( \lambda_{k} \right) _{k \geq 0}$ be a given positive sequence satisfying  $\displaystyle{\sum_{k \geq 0} \lambda_{k} =+\infty}$, $\tau_0=0$ and $\tau_k =  \sum_{i = 0}^{k-1} \lambda_{i}$ for every $k \geq 1$. Then for any sequence  $\left( y_{k} \right) _{k \geq 0}$  generated by the proximal algorithm

\begin{equation}
	\label{im prox}
	y_{k+1} := \prox_{\lambda_{k} f} \left( y_{k} \right) \quad \forall k \geq 0,
\end{equation}
 the following properties are satisfied:
\begin{enumerate}
	\item \label{thm:prox1:sf} (summability of function values) \quad
	$\displaystyle{\sum_{k \geq 0} \lambda_{k} \left( f \left( y_{k+1} \right) - \inf\nolimits_{\cH} f \right) < + \infty };$
	
	\item \label{thm:prox1:sg} (summability of gradients) \quad
	$\displaystyle{\sum_{k \geq 0} \tau_k \lambda_{k} \left\lVert \nabla f \left( y_{k+1} \right) \right\rVert ^{2} < + \infty };$
	
	\item \label{thm:prox1:sv} (summability of velocities) \quad
	$\displaystyle{\sum_{k \geq 0} \dfrac{\tau_k}{\lambda_{k}} \left\lVert y_{k+1} - y_{k} \right\rVert ^{2} < + \infty} ;$
	
	\item \label{thm:prox1:func} (convergence of function values) \quad
	$\displaystyle{f \left( y_{k+1} \right) - \inf\nolimits_{\cH} f = o \left( \frac{1}{\tau_{k+1}} \right)}$ as $k \to + \infty ;$			
	
	\item \label{thm:prox1:grad} (convergence of gradient) \quad
	$\displaystyle{ \left\lVert \nabla f \left( y_{k+1} \right) \right\rVert = o \left( \frac{1}{\sqrt{\sum_{l = 0}^{k} \tau_l \lambda_{l}}} \right)}$ as $k \to + \infty ;$
	
	\item \label{thm:prox1:seq} the sequence of iterates $\left( y_{k} \right) _{k \geq 0}$ converges weakly as $k \to + \infty$, and its limit belongs to $S = \argmin_{\cH} f$.
\end{enumerate}
\end{theorem}
\begin{proof}
Let $k \geq 0$ be fixed. Take $z_{*} \in S = \argmin f$. According to \eqref{im 1e} and the convexity of $f$, we deduce that
\begin{align}
	\dfrac{1}{2} \left\lVert y_{k+1} - z_{*} \right\rVert ^{2} & = \dfrac{1}{2} \left\lVert y_{k} - z_{*} \right\rVert ^{2} + \left\langle y_{k+1} - z_{*} , y_{k+1} - y_{k} \right\rangle - \dfrac{1}{2} \left\lVert y_{k+1} - y_{k} \right\rVert ^{2} \nonumber \\
	& = \dfrac{1}{2} \left\lVert y_{k} - z_{*} \right\rVert ^{2} - \lambda_{k} \left\langle y_{k+1} - z_{*} , \nabla f \left( y_{k+1} \right) \right\rangle - \dfrac{1}{2} \lambda_{k}^{2} \left\lVert \nabla f \left( y_{k+1} \right) \right\rVert ^{2} \nonumber \\
	& \leq \dfrac{1}{2} \left\lVert y_{k} - z_{*} \right\rVert ^{2} - \lambda_{k} \left( f \left( y_{k+1} \right) - \inf\nolimits_{\cH} f \right) - \dfrac{1}{2} \lambda_{k}^{2} \left\lVert \nabla f \left( y_{k+1} \right) \right\rVert ^{2} . \label{prox-al:04}
\end{align}
Statement \ref{thm:prox1:sf} follows from \cite[Lemma 5.31]{BaCo}. In addition, the limit $\lim_{k \to + \infty} \left\lVert y_{k} - z_{*} \right\rVert \in \R$ exists, which means that the first condition of the discrete Opial's lemma is fulfilled. 

On the other hand, the sequence $\left( f \left( y_{k} \right) - \inf\nolimits_{\cH} f \right) _{k \geq 0}$ is nonincreasing. Precisely,  we have for every $k \geq 0$ 
\begin{equation}\label{prox-al:05}
	\left(f \left( y_{k} \right) - \inf\nolimits_{\cH} f \right) - \left(f \left( y_{k+1} \right) - \inf\nolimits_{\cH} f \right) \geq \langle \nabla f \left( y_{k+1} \right) , y_k-y_{k+1} \rangle  = \lambda_{k} \left\lVert \nabla f \left( y_{k+1} \right) \right\rVert ^{2} \geq 0.
\end{equation}
According to \cite[Lemma 22]{AC2} we get
\begin{equation*}
	f \left( y_{k+1} \right) - \inf\nolimits_{\cH} f = o \left( \dfrac{1}{\sum_{i=0}^{k} \lambda_{i}} \right) ,
\end{equation*}
which proves \ref{thm:prox1:func}.

Let $k \geq 1$.  Multiplying both sides of \eqref{prox-al:05} by $\tau_k=\sum_{i = 0}^{k-1} \lambda_{i} > 0$, then adding the result into \eqref{prox-al:04}, we get 
\begin{align*}
	\tau_{k+1} \left( f \left( y_{k+1} \right) - \inf\nolimits_{\cH} f \right) + \dfrac{1}{2} \left\lVert y_{k+1} - z_{*} \right\rVert ^{2} 
	& \leq \tau_k \left(f \left( y_{k} \right) - \inf\nolimits_{\cH} f \right) + \dfrac{1}{2} \left\lVert y_{k} - z_{*} \right\rVert ^{2} \nonumber \\
	& \qquad - \dfrac{1}{2} \lambda_{k}^{2} \left\lVert \nabla f \left( y_{k+1} \right) \right\rVert ^{2} - \tau_k \lambda_{k} \left\lVert \nabla f \left( y_{k+1} \right) \right\rVert ^{2} .
\end{align*}
This implies 
\begin{equation*}
	\sum_{k \geq 0} \lambda_{k}^{2} \left\lVert \nabla f \left( y_{k+1} \right) \right\rVert ^{2} < + \infty
	\quad \textrm{ and } \quad
	\sum_{k \geq 1} \tau_k \lambda_{k} \left\lVert \nabla f \left( y_{k+1} \right) \right\rVert ^{2} < + \infty ,
\end{equation*}
which yields \ref{thm:prox1:sg}. From \eqref{im 1e}, we infer \ref{thm:prox1:sv}.  To deduce \ref{thm:prox1:grad}, it  suffices to show that the sequence $\left( \|\nabla f \left( y_{k} \right) \| \right) _{k \geq 0}$ is nonincreasing.
Indeed, it follows from \eqref{im prox} and the cocoercivity of $\nabla f$ that
\begin{align*}
	\dfrac{1}{2} \left\lVert \nabla f \left( y_{k+1} \right) \right\rVert ^{2}
	& = \dfrac{1}{2} \left\lVert \nabla f \left( y_{k} \right) \right\rVert ^{2} + \left\langle \nabla f \left( y_{k+1} \right) , \nabla f \left( y_{k+1} \right) - \nabla f \left( y_{k} \right) \right\rangle - \dfrac{1}{2} \left\lVert \nabla f \left( y_{k+1} \right) - \nabla f \left( y_{k} \right) \right\rVert ^{2} \nonumber \\
	& = \dfrac{1}{2} \left\lVert \nabla f \left( y_{k} \right) \right\rVert ^{2} - \dfrac{1}{\lambda_{k}} \left\langle y_{k+1} - y_{k} , \nabla f \left( y_{k+1} \right) - \nabla f \left( y_{k} \right) \right\rangle - \dfrac{1}{2} \left\lVert \nabla f \left( y_{k+1} \right) - \nabla f \left( y_{k} \right) \right\rVert ^{2} \nonumber \\
	& \leq \dfrac{1}{2} \left\lVert \nabla f \left( y_{k} \right) \right\rVert ^{2} .
\end{align*}
Taking into account also \ref{thm:prox1:sg}, we obtain \ref{thm:prox1:grad}.

\noindent Finally, according to the assumption  $\sum_{k \geq 0} \lambda_{k} =+\infty$, and  $(iv)$, we have that  $\lim_{k \to + \infty} f \left( y_{k} \right) = \inf\nolimits_{\cH} f$. Since $f$ is convex and lower semicontinuous, the second condition of Opial's lemma is also fulfilled. This gives the weak convergence of the sequence $\left( y_{k} \right) _{k \geq 0}$  to an element in $S=\argmin f$.
\qed
\end{proof}

Then,  we give a statement which can be seen as a discrete counterpart of Lemma \ref{lemma tau}. The result is more complex not only because we are in the discrete setting, but also because it allows  an explicit choice of the stepsize, as we will see later.

\begin{lemma}\label{lemma tau-k}
	Let $\left( \lambda_{k} \right) _{k \geq 0}$ be a positive sequence and $\left( \tau_{k} \right) _{k \geq 0}$ such that $\tau_0=0$ and $\tau_k =  \sum_{i = 0}^{k-1} \lambda_{i}$ for all $k \geq 1$. Suppose that there exist $C_{2}>0$ and $a, b, c \geq 0$ such that $b + c > a$ and
	\begin{equation*}
		\sum_{k \geq 0} \tau_{k}^{a} \lambda_{k}^{-b} \lambda_{k+1}^{-c} \leq C_{2} < + \infty
	\end{equation*}
	Then there exists $C_{3} > 0$ such that for every $k \geq 1$ it holds
	\begin{equation}\label{est tau-k}
		\tau_{k+1} \geq C_{3} k^{\frac{b+c+1}{b+c-a}} .
	\end{equation}
\end{lemma}
\begin{proof}
By applying the H\"{o}lder inequality twice we get for all $k \geq 0$
\begin{align}
	\sum_{i = 0}^{k} \tau_{i}^{\frac{a}{b+c+1}} & \leq \left( \sum_{i = 0}^{k} \tau_{i}^{a} \lambda_{i}^{-b} \lambda_{i+1}^{-c} \right) ^{\frac{1}{b+c+1}} \left( \sum_{i = 0}^{k} \lambda_{i} \right) ^{\frac{b}{b+c+1}} \left( \sum_{i = 0}^{k} \lambda_{i+1} \right) ^{\frac{c}{b+c+1}} \nonumber \\
	& \leq C_{2}^{\frac{1}{b+c+1}} \left( \sum_{i = 0}^{k+1} \lambda_{i} \right) ^{\frac{b+c}{b+c+1}} = C_{2}^{\frac{1}{b+c+1}} \tau_{k+2} ^{\frac{b+c}{b+c+1}}. \label{est Holder-3} 
\end{align}

\noindent If $a=0$ then \eqref{est tau-k} follows immediately.
From now on we  suppose that $a > 0$. Inequality \eqref{est Holder-3} becomes
\begin{equation}
	\label{est Holder-seq}
	\sum_{i = 0}^{k} \tau_{i}^{\frac{a}{b+c+1}} \leq C_{2}^{\frac{1}{b+c+1}} \left( \tau_{k+2} ^{\frac{a}{b+c+1}} \right) ^{\frac{b+c}{a}} \quad \forall k \geq 0.
\end{equation}
Following the continuous counterpart, let us define
\begin{equation*}
	C_{b+c} := C_{2}^{\frac{1}{b+c+1}} > 0 
	\qquad \textrm{ and } \qquad
	A_{k} := \sum_{i = 0}^{k} \tau_{i}^{\frac{a}{b+c+1}} \qquad \forall k \geq 0
\end{equation*}
so that \eqref{est Holder-seq} becomes
\begin{equation*}
	A_{k} \leq C_{b+c} \left( A_{k+2} - A_{k+1} \right) ^{\frac{b+c}{a}} \quad \forall k \geq 0.
\end{equation*}
From here,

\begin{equation}
	\label{est rec}
	C_{b+c}^{-\frac{a}{b+c}} \leq A_{k}^{- \frac{a}{b+c}} \left( A_{k+2} - A_{k+1} \right) \quad \forall k \geq 1.
\end{equation}

\noindent For convenience, we define the following function $\psi \colon \R_{++} \to \R_{++}$ as $\psi \left( r \right) := r^{-\frac{a}{b+c}}$. It is clear that

\begin{equation*}
	\dfrac{d}{dr} \left( \dfrac{b+c}{b+c-a} r^{1-\frac{a}{b+c}} \right) = \psi \left( r \right) \qquad \textrm{ and } \qquad \dot{\psi} \left( r \right) = - \dfrac{a}{b+c} r^{-\frac{a}{b+c}-1} < 0 .
\end{equation*}

\smallskip

\noindent Since $\left( A_{k} \right) _{k \geq 0}$ is increasing, this means $\psi \left( A_{k+2} \right) \leq \psi \left( r \right) \leq \psi(A_k) $ for every $A_{k} \leq r \leq A_{k+2}$.

\smallskip

\noindent Let $k \geq 1$ fixed. We consider two separate cases.

\item[Case 1: $\psi \left( A_{k} \right) \leq 2 \psi \left( A_{k+2} \right)$.] Then \eqref{est rec} leads to
\begin{align*}
	C_{b+c}^{-\frac{a}{b+c}} & \leq A_{k}^{- \frac{a}{b+c}} \left( A_{k+2} - A_{k} \right) = \psi \left( A_{k} \right) \left( A_{k+2} - A_{k} \right) \nonumber \\
	& \leq 2 \psi \left( A_{k+2} \right) \left( A_{k+2} - A_{k} \right) = 2 \psi \left( A_{k+2} \right) \int_{A_{k}}^{A_{k+2}} 1dr \nonumber \\
	& \leq 2 \int_{A_{k}}^{A_{k+2}} \psi \left( r \right) dr = 2\dfrac{b+c}{b+c-a} \left( A_{k+2}^{1-\frac{a}{b+c}} - A_{k}^{1-\frac{a}{b+c}} \right) .
\end{align*}

\smallskip

\item[Case 2: $\psi \left( A_{k} \right) > 2 \psi \left( A_{k+2} \right)$.]
This is equivalent to $A_{k+2} > 2^{\frac{b+c}{a}} A_{k}$.
Since $b+c>a$, we can deduce further

\begin{equation*}
	A_{k+2}^{1-\frac{a}{b+c}} > 2^{\frac{b+c}{a}-1} A_{k}^{1-\frac{a}{b+c}} .
\end{equation*}
Consequently,
\begin{equation*}
	A_{k+2}^{1-\frac{a}{b+c}} - A_{k}^{1-\frac{a}{b+c}} > \left( 2^{\frac{b+c}{a}-1} - 1 \right) A_{k}^{1-\frac{a}{b+c}} \geq \left( 2^{\frac{b+c}{a}-1} - 1 \right) A_{1}^{1-\frac{a}{b+c}} ,
\end{equation*}
recall that the last inequality follows from the increasing property of $\left( A_{k} \right) _{k \geq 1}$.

\smallskip

\noindent In conclusion, for every $k \geq 0$ we have
\begin{equation*}
	A_{k+2}^{1-\frac{a}{b+c}} - A_{k}^{1-\frac{a}{b+c}} \geq C_{4} := \min \left\lbrace \frac{1}{2}\left( 1 - \dfrac{a}{b+c} \right) C_{b+c}^{-\frac{a}{b+c}} ,  \left( 2^{\frac{b+c}{a}-1} - 1 \right) A_{1}^{1-\frac{a}{b+c}}, A_{2}^{1-\frac{a}{b+c}} \right\rbrace > 0 .
\end{equation*}
Telescoping sum arguments combined with \eqref{est Holder-seq} imply for every $k \geq 1$
\begin{equation*}
	C_{4}k \leq A_{2k}^{1-\frac{a}{b+c}} - A_{0}^{1-\frac{a}{b+c}} \leq A_{2k}^{1-\frac{a}{b+c}} \leq C_{2}^{\frac{b+c-a}{(b+c)(b+c+1)}} \tau_{2k+2} ^{\frac{b+c-a}{b+c+1}} .
\end{equation*}
This  gives for every $k \geq 1$
$$
\tau_{2k+3} \geq \tau_{2k+2} \geq \widetilde C_{3} k^{\frac{b+c+1}{b+c-a}},
$$
where $\widetilde C_3 >0$.
We therefore deduce that there exists $C_3 >0$ such that
$$
\tau_{k+1} \geq C_3 k^{\frac{b+c+1}{b+c-a}} \quad \forall k \geq 1,
$$
which gives \eqref{est tau-k}.
\qed
\end{proof}

Following a plan identical to the continuous case, we successively consider the case where the control by feedback is formulated in terms of speed, then of gradient.

\subsection{Adaptive stepsize rules resulting from the discretization of the velocity based system}

In this subsection we specialize the algorithm (PEAS) to the case where 
$G(y_{k} , y_{k-1} )= \left\lVert y_{k} - y_{k-1} \right\rVert$.
\begin{algorithm}[h!]	
	\caption{Proximal algorithm with adaptive step size defined via velocity}
	\label{algo:prox-vbased}
	\vspace{2mm}
	\DontPrintSemicolon 
	
	\KwIn{$y_{0} \neq y_{-1} \in \cH$}
	\For{$k = 0, 1, \cdots$}
	{
		\uIf{$\nabla f \left( y_{k} \right) = 0$}{
			$\textbf{stop}$\;
		}
		\Else{
			$\lambda_{k} := \left\lVert y_{k} - y_{k-1} \right\rVert ^{-\frac{p-1}{p}}$\\
			$y_{k+1}	 := \prox_{\lambda_{k} f} \left( y_{k} \right)$\\
		}
	}
\end{algorithm}


\begin{theorem}\label{thm:prox-adapt-vel}
	Let $\left( y_{k} \right) _{k \geq 0}$ be the sequence generated by Algorithm {\rm \ref{algo:prox-vbased}}.
	Then it holds
	\begin{equation*}
		f \left( y_{k} \right) - \inf\nolimits_{\cH} f = o \left( \dfrac{1}{k^{2-\frac{1}{p}}} \right) \textrm{ as } k \to + \infty,
	\end{equation*}
	and the sequence of iterates $\left( y_{k} \right) _{k \geq 0}$ converges weakly as $k \to + \infty$, and its limit belongs to $S = \argmin_{\cH} f$.
\end{theorem}
\begin{proof}
	By the choice of the step size, we have from Theorem \ref{thm:prox1} \ref{thm:prox1:sv} that
	\begin{equation*}
		\sum_{k \geq 0} \dfrac{\tau_k}{\lambda_{k}} \left\lVert y_{k+1} - y_{k} \right\rVert ^{2} = \sum_{k \geq 0} \tau_k \lambda_{k}^{-1} \lambda_{k+1}^{-\frac{2p}{p-1}} < + \infty,
	\end{equation*}
where $\tau_0=0$ and $\tau_{k} := \sum_{i=0}^{k-1} \lambda_{i}$ for every $k \geq 1$. We are in position to apply Lemma \ref{lemma tau-k} with $\left( a , b , c \right) := \left( 1 , 1 , \frac{2p}{p-1} \right)$. We get 
\begin{equation}\label{eq:30-12-22}
\tau_{k+1} \geq C_{3} k^{2- \frac{1}{p}} \quad \forall k \geq 1. 
\end{equation}	
Therefore 	$\sum_{k \geq 0} \lambda_{k}= \lim_{k \rightarrow +\infty} \tau_{k} =+\infty$, and we can apply Theorem \ref{thm:prox1} to obtain
the conclusion.
	\qed
\end{proof}

\subsection{Adaptive stepsize resulting from the discretization of the gradient based system}

Now let us specialize the algorithm (PEAS) to the case where 
$G(y_{k} , y_{k-1} )= \left\lVert \nabla f (y_{k}) \right\rVert$.

\begin{algorithm}[h!]	
	\caption{Proximal algorithm with adaptive step size defined via gradient}
	\label{algo:prox-gbased}
	\DontPrintSemicolon 
	\KwIn{$y_{0} \in \cH$}
	\For{$k = 0, 1, \cdots$}
	{
		\uIf{$\nabla f \left( y_{k} \right) = 0$}{
			$\textbf{stop}$\;
		}
		\Else{
			$\lambda_{k} := \left\lVert \nabla f \left( y_{k} \right) \right\rVert ^{-\frac{p-1}{p}}$\\
			$y_{k+1}	 := \prox_{\lambda_{k} f} \left( y_{k} \right)$\\
		}
	}
\end{algorithm}

\begin{theorem}
	Let $\left( y_{k} \right) _{k \geq 0}$ be the sequence generated by Algorithm {\rm \ref{algo:prox-gbased}}.
	Then it holds
	\begin{equation*}
		f \left( y_{k+1} \right) - \inf\nolimits_{\cH} f = o \left( \frac{1}{k^{ 2- \frac{1}{p}}} \right) \textrm{ as } k \to + \infty
	\end{equation*}
	and the sequence of iterates $\left( y_{k} \right) _{k \geq 0}$ converges weakly as $k \to + \infty$, and its limit belongs to $S = \argmin_{\cH} f$.
\end{theorem}
\begin{proof}
In this case  we have from Theorem \ref{thm:prox1} \ref{thm:prox1:sg}

\begin{equation}\label{ineq_k_k+1}
	\sum_{k \geq 0} \tau_{k} \lambda_{k} \left\lVert \nabla f \left( y_{k+1} \right) \right\rVert ^{2} = \sum_{k \geq 0} \tau_{k} \lambda_{k} \lambda_{k+1}^{-\frac{2p}{p-1}} < + \infty,
\end{equation}
where $\tau_0=0$ and $\tau_{k} := \sum_{i=0}^{k-1} \lambda_{i}$ for every $k \geq 1$.

\noindent  Let us establish an inequality of the type

\begin{equation*}
	\left\lVert \nabla f \left( y_{k} \right) \right\rVert \leq C_k \left\lVert \nabla f \left( y_{k+1} \right) \right\rVert ,
\end{equation*}
for some sequence $C_k >0$ which is to be precised. We have for all $k \geq 0$
\begin{align}
	\left\lVert \nabla f \left( y_{k} \right) \right\rVert ^{2} & = \left\lVert \nabla f \left( y_{k+1} \right) \right\rVert ^{2} - 2 \left\langle \nabla f \left( y_{k+1} \right) , \nabla f \left( y_{k+1} \right) - \nabla f \left( y_{k} \right) \right\rangle + \left\lVert \nabla f \left( y_{k+1} \right) - \nabla f \left( y_{k} \right) \right\rVert ^{2} \nonumber \\
	& = \left\lVert \nabla f \left( y_{k+1} \right) \right\rVert ^{2} + \dfrac{2}{\lambda_{k}} \left\langle y_{k+1} - y_{k} , \nabla f \left( y_{k+1} \right) - \nabla f \left( y_{k} \right) \right\rangle + \left\lVert \nabla f \left( y_{k+1} \right) - \nabla f \left( y_{k} \right) \right\rVert ^{2} \nonumber \\
	& \leq \left\lVert \nabla f \left( y_{k+1} \right) \right\rVert ^{2} + \left( \dfrac{2L}{\lambda_{k}} + L^{2} \right) \left\lVert y_{k+1} - y_{k} \right\rVert ^{2} = \left( 1 + L \lambda_{k} \right) ^{2} \left\lVert \nabla f \left( y_{k+1} \right) \right\rVert ^{2} \nonumber \\
	& \leq  \left( 1 + L \lambda_{k+1} \right) ^{2} \left\lVert \nabla f \left( y_{k+1} \right) \right\rVert ^{2} \label{ineq_k_k+1-2}
\end{align}
where $L >0$ denotes the Lipschitz constant of $\nabla f$ on a bounded set containing the sequence $\left( y_{k} \right) _{k \geq 0}$. 

\noindent Combining \eqref{ineq_k_k+1} and \eqref{ineq_k_k+1-2}, we get
\begin{equation}\label{ineq_k_k+1-4}
	\sum_{k \geq 0} \tau_{k} \lambda_{k} \frac{1}{\left( 1 + L \lambda_{k+1} \right) ^{2}}
	\left\lVert \nabla f \left( y_{k} \right) \right\rVert ^{2} = \sum_{k \geq 0} \tau_{k} \lambda_{k} \frac{1}{\left( 1 + L \lambda_{k+1} \right) ^{2}} \lambda_{k}^{-\frac{2p}{p-1}} < + \infty.
\end{equation}

Let us now show that $\lim_{k \rightarrow +\infty} \lambda_{k}=+\infty$. 
According to the decreasing property of the  sequence $\left( f \left( y_{k} \right) - \inf\nolimits_{\cH} f \right) _{k \geq 0}$, by summing inequalities \eqref{prox-al:05}  we get
\begin{equation}\label{prox-al:055}
	\sum_{k \geq 0}  \lambda_{k} \left\lVert \nabla f \left( y_{k+1} \right) \right\rVert ^{2} < +\infty.
\end{equation}
From the closed-loop rule we deduce that 
\begin{equation}\label{prox-al:056}
	\sum_{k \geq 0}  \lambda_{k} \lambda_{k+1}^{-\frac{2p}{p-1}} < +\infty.
\end{equation}
Therefore
$$
\lim_{k \rightarrow +\infty} \lambda_{k} \lambda_{k+1}^{-\frac{2p}{p-1}} =0.
$$
Since $(\lambda_{k})_{k \geq 0}$ is increasing, let us denote by $l >0$ its limit.
If $l$ is finite then, by passing to the limit on the above inequality we get $l^{1-\frac{2p}{p-1}}=0$,  a clear contradiction with $l >0 $. Therefore
$$
\lim_{k \rightarrow +\infty} \lambda_{k}  = +\infty.
$$

\noindent In this case $\frac{1}{\left( 1 + L \lambda_{k+1} \right)^{2}} \sim 
\left(  L \lambda_{k+1} \right)^{-2}$, which gives
\begin{equation}\label{ineq_k_k+1-44}
	\sum_{k \geq 0} \tau_{k} \lambda_{k}^{1-\frac{2p}{p-1}} 
	 {\lambda_{k+1}}^{-2} < + \infty.
\end{equation}

\noindent We are in position to apply Lemma \ref{lemma tau-k} with $\left( a , b , c \right) := \left( 1 , \frac{2p}{p-1} -1 , 2 \right)$.
We get 
$$
\tau_{k+1} \geq C_{3} k^{ 2- \frac{1}{p}} \quad \forall k \geq 1. 
$$	
We have $\sum_{k \geq 0} \lambda_{k}= \lim_{k \rightarrow +\infty} \tau_{k} =+\infty$, and we can apply Theorem \ref{thm:prox1} to obtain,
as $k \to + \infty $
$$f \left( y_{k+1} \right) - \inf\nolimits_{\cH} f = o \left( \frac{1}{k^{ 2- \frac{1}{p}}} \right).
$$ 
This completes the proof. \qed
\end{proof}

\begin{remark}{\rm
Note that the closed-loop control of the velocity and the closed-loop control of the gradient give the same convergence rate of the values.
Clearly, we have obtained a faster convergence result compared to the classical Proximal Point Algorithm (PPA), which we also cover since it coincides with $p=1$ in both cases.}
\end{remark}

\section{Inertial proximal algorithms obtained by closed-loop damping}

Let us now consider the convergence properties of  the sequences $(x_k)_{k \geq 0}$ which are obtained by applying  the averaging process to the sequences generated by Algorithm \ref{algo:prox-vbased}.  Indeed, we limit our investigation to the closed loop control of the velocity, the case of the closed loop control of the gradient is very similar.
Let us discretize the continuous averaging relation
$$\dot{x}(t) + \dfrac{\dot{\tau}(t)}{\tau(t)} \left( x(t) - y(t) \right)= 0
$$

\noindent as follows (recall that, because of the choice $q=1$, we have 
$\dot{\tau}(t)= \lambda (t)$)
$$
x_{k+1} - x_{k} + \frac{\lambda_k}{\tau_{k+1}} \left( x_k -  y_{k+1}   \right)=0.
$$

\noindent Equivalently 

$$
x_{k+1} = \left( 1- \frac{\lambda_k}{\tau_{k+1}}\right)	  x_k  + \frac{\lambda_k}{\tau_{k+1}}  y_{k+1}. 
$$

\noindent This gives the following proximal inertial algorithm:

\begin{algorithm}[h!]	
	\caption{Proximal inertial algorithm with adaptive step size defined via velocity}
	\label{algo:prox-inertial}
	\DontPrintSemicolon 
	\KwIn{$\tau_{0} := 0$ and $x_0, y_{0} \neq y_{-1} \in \cH$}
	\For{$k = 0, 1, \cdots$}
	{
		\uIf{$\nabla f \left( y_{k} \right) = 0$}{
			$\textbf{stop}$\;
		}
		\Else{
			$\displaystyle{\lambda_{k} := \left\lVert y_{k} - y_{k-1} \right\rVert ^{-\frac{p-1}{p}}}$ \vspace{1mm}\\
			$\displaystyle{ y_{k+1}	 := \prox_{\lambda_{k} f} \left( y_{k} \right)}$ \vspace{1mm}
			\\
$\displaystyle{\tau_{k+1} := \tau_k + \lambda_k}$ \vspace{1mm}\\
			$\displaystyle{x_{k+1} : = \left( 1- \frac{\lambda_k}{\tau_{k+1}}\right)	  x_k  + \frac{\lambda_k}{\tau_{k+1}}  y_{k+1}.}$}
	}		
\end{algorithm}

\begin{theorem}\label{thm:prox-conv-31}
Let $\left( x_{k} \right) _{k \geq 0}$ be the sequence generated by Algorithm {\rm \ref{algo:prox-inertial}}.
Then it holds
$$f \left( x_{k} \right) - \inf\nolimits_{\cH} f = \mathcal O \left( \frac{1}{k^{2-\frac{1}{p}}} \right) \textrm{ as } k \to + \infty$$
and the sequence of iterates $\left(x_{k} \right) _{k \geq 0}$ converges weakly as $k \to + \infty$, and its limit belongs to $S = \argmin_{\cH} f$.
\end{theorem}
		
\begin{proof}
Let $k \geq 0$. By definition of $x_{k+1}$ we have
$$
\tau_{k+1} x_{k+1} = (\tau_{k+1}-\lambda_k )x_k + \lambda_{k} y_{k+1} 
$$
which gives
(recall that $\tau_{k+1} = \sum_{i=0}^k   \lambda_i$)
$$
\tau_{k+1} x_{k+1}   - \tau_{k} x_{k}= (\tau_{k+1}-\lambda_k )x_k + \lambda_{k} y_{k+1} - \tau_{k} x_{k}= ( \tau_{k+1} - \tau_{k} -\lambda_k )x_k +\lambda_{k} y_{k+1}  = \lambda_{k} y_{k+1}.
$$

\noindent Therefore, by telescoping arguments we obtain
\begin{equation}\label{eq:average}
 x_{k+1} = \frac{\sum_{i=0}^{k} \lambda_i y_{i+1}} {\tau_{k+1}} \quad \forall k \geq 0.
\end{equation}
By convexity of $f$ we infer
\begin{align*}
	f \left( x_{k+1} \right) - \inf\nolimits_{\cH} f
	& = \left( f - \inf\nolimits_{\cH} f \right) \left( x_{k+1} \right) = \left( f - \inf\nolimits_{\cH} f \right) \left( \frac{\sum_{i=0}^{k} \lambda_i y_{i+1}} {\tau_{k+1}} \right) \nonumber \\
	& \leq \dfrac{1}{\tau_{k+1}} \sum_{i=0}^{k} \lambda_i \left( f - \inf\nolimits_{\cH} f \right) \left( y_{i+1} \right) = \dfrac{1}{\tau_{k+1}} \sum_{i=0}^{k} \lambda_i \left( f \left( y_{i+1} \right) - \inf\nolimits_{\cH} f \right) ,
\end{align*}
By Theorem \ref{thm:prox1} \ref{thm:prox1:sf}, we have $\sum_{k \geq 0} \lambda_{k} \left( f \left( y_{k+1} \right) - \inf\nolimits_{\cH} f \right) < + \infty ,$ and by \eqref{eq:30-12-22} we have $\tau_{k+1} \geq k^{2-\frac{1}{p}}$, which gives the claim.
The weak convergence of $\left(x_{k} \right) _{k \geq 0}$ to an element in $S = \argmin_{\cH} f$ follows from the weak convergence of $\left(y_{k} \right) _{k \geq 0}$ and the Stolz-Ces\'aro Theorem.
\qed 
\end{proof}

\subsection{Geometric interpretation of Algorithm \ref{algo:prox-inertial}}

First note that Algorithm \ref{algo:prox-inertial} can be equivalently written as follows
\begin{equation}
 x_{k+1} = \left( 1- \frac{\lambda_k}{\tau_{k+1} }\right) x_k +  \frac{\lambda_k}{\tau_{k+1}} \prox_{\lambda_{k} f} \left( x_{k-1} + \frac{\tau_{k} }{\lambda_{k-1}} (x_k -  x_{k-1} )\right). 
\end{equation}
Since  $\frac{\tau_{k} }{\lambda_{k-1}} >1$, the algorithm first involves an extrapolation step (this is the inertial aspect), then a proximal step, and finally a relaxation step which balances the inertia effect and dampens the oscillations. This is shown in the figure below.
We set $\theta_k = \frac{\lambda_k}{\tau_{k+1}} \in ]0, 1[$.

\begin{figure}[!h] 
\setlength{\unitlength}{8cm}
\begin{picture}(0.5,0.7)(-0.65,-0.05)
\tcb{
\put(0.357,0.51){$y_k =  x_{k-1} +  \frac{1}{\theta_{k-1}} ( x_{k}  - x_{k-1})$}
\put(0.328,0.493){{\tiny $\bullet$}}
\put(0.369,0.582){$x_k$}
\put(0.348,0.565){{\tiny $\bullet$}}
\put(0.37,0.657){$x_{k-1}$}
\put(0.364,0.635){{\tiny $\bullet$}}
\put(0.273,0.485){{\tiny $\bullet$}}
\put(-0.4,0.485){$ x_{k+1} = \left( 1- \theta_k \right) x_k +  \theta_k \prox_{\lambda_{k} f} \left( y_k \right)$}
\put(0.255,0.388){$\prox_{\lambda_{k} f}(y_k) $}
\put(-0.08,0.29){$S$}
}
\qbezier(-0.1,0.2)(0.4,0.37)(-0.1,0.42)
\qbezier(-0.1,0.15)(0.57,0.4)(-0.1,0.48)
\qbezier(-0.1,0.02)(1.05,0.45)(-0.1,0.6)
\qbezier(0.,-0.01)(1.2,0.53)(-0.1,0.65)
\qbezier(0.16,-0.01)(1.42,0.59)(-0.1,0.709)
\put(-0.1,0.23){\line(1,1){0.16}}
\put(-0.02,0.23){\line(1,1){0.136}}
\put(-0.1,0.32){\line(1,1){0.085}}

\put(0.019,0.633){\line(5,-2){0.6}}
\put(0.34,0.504){\vector(-4,-3){0.13}}

\put(0.325,0.505){\line(1,4){0.035}}

\put(0.174,0.4){\line(5,6){0.145}}

\put(-0.2,-0.1){Fig.1 \, A geometrical illustration of Algorithm \ref{algo:prox-inertial}}
\end{picture}
\label{fig1NAG}
\end{figure}

\vspace{3mm}

Despite some analogies,  Algorithm \ref{algo:prox-inertial} is different from the relaxed inertial proximal algorithm (RIPA)  considered by Attouch and Cabot in \cite{AC-RIPA}, and which writes

\begin{equation}
	\label{def:RIPA}
	\begin{cases}
		y_k &=  x_{k} +  \alpha_{k} ( x_{k}  - x_{k-1}) \vspace{1mm}\\		
		x_{k+1} &= (1-\rho_k)y_k + \rho_k \prox_{\lambda_{k} f}(y_k)   
	\end{cases}
\end{equation}

As main difference,  in Algorithm \ref{algo:prox-inertial} the relaxation is taken between $x_k$ and $\prox_{\lambda_{k} f}(y_k)$, while in (RIPA) it is taken between $y_k$ and $\prox_ { \lambda_{k} f}(y_k)$. Consequently,  Algorithm \ref{algo:prox-inertial} involves a Hessian damping effect which is not present in (RIPA).  Note in Algorithm \ref{algo:prox-inertial} the balance between the extrapolation (inertial, acceleration) effect and the relaxation effect. Moreover, our construction provides coefficients which are generated automatically in closed loop way, whereas in (RIPA) they require subtle adjustment.
The importance of the relaxation technique when combined with inertia has been put to the fore in \cite{IH}.
According to \eqref{eq:average}, $x_{k+1}$ can be interpreted as an average of the $\{y_n:  0 \leq n \leq k+1 \}$, which makes our approach somewhat analogous to the nonlinear averaging technique developed in \cite{SAB}, where it is assumed that there is a unique minimizer.  Averaging techniques have also been used in \cite{PL} in the context of hybrid systems. Indeed, adjusting  the damping in a closed-loop  ad hoc manner bears  some analogy to restarting methods.

\subsection{Extension to the nonsmooth setting}

Since our proposed numerical algorithms are proximal methods,  they can be used also to minimize nonsmooth and convex functions. Given an optimization problem
\begin{equation}\label{pb:nonsmooth}
	\min \left\lbrace  f(x) : \, x\in\cH  \right\rbrace,
\end{equation}
where $f: \cH \rightarrow \R \cup \left\lbrace + \infty \right\rbrace$ is a proper, lower semicontinuous, and convex function with $S = \argmin_{\cH} f \neq \emptyset$,  one can equivalently consider the problem
\begin{equation}\label{pb:Moreau}
	\min \left\lbrace  f_{\gamma}(x) : \, x\in\cH  \right\rbrace,
\end{equation}
where $f_{\gamma} : \cH \rightarrow \R$ stands for the Moreau envelope of $f$ with parameter $\gamma > 0$,  defined as 
\begin{equation*}
	f_{\gamma} \left( x \right) = \inf_{y \in \cH} \left\lbrace  f \left( x \right) + \dfrac{1}{2 \gamma} \left\lVert x - y \right\rVert ^{2} \right\rbrace .
\end{equation*}
The two problems share the same optimal value and solution set $\argmin_{\cH} f = \argmin_{\cH} f_{\gamma}$,  while the Moreau envelope is convex and differentiable and it has a $\gamma^{-1}$-Lipschitz continuous gradient (\cite{BaCo,Beck}). The formulas for the proximal operators of $f$ and $f_{\gamma}$ are closely related by a simple convex combination, which makes no difference when solving \eqref{pb:nonsmooth} and \eqref{pb:Moreau} conceptually.  We can therefore exploit these premises to apply our methods to a broader class of functions that are only assumed to be proper, lower semicontinuous, and convex.

\section{Numerical experiments}

In this section, we carry out numerical experiments in order to demonstrate the effectiveness of the methods we have proposed.

Let $q \geq 1$ and $X \in \R^{m \times n}$.
We denote by $\norm{\cdot}_{\mathcal{S}_{q}}$ the Schatten $q$-norm of $X$,  which is defined as
\begin{equation*}
\norm{X}_{\mathcal{S}_{q}}^{q} = \sum_{i} \sigma_{i}^{q} \left( X \right) ,
\end{equation*}
where $\sigma_{i} \left( X \right)$ denote the $i^{th}$-singular value of $X$.  In case $q := 1$ it gives the nuclear norm,  which we denote by $\norm{\cdot}_{\ast}$,  and in case $q := 2$ it gives the Frobenius norm, which we denote by $\norm{\cdot}_{\mathsf{F}}$.  Further,  we denote by $E_{\lambda_{-},\lambda_{+}} \left( \R^{n \times n} \right)$  the set of positive semidefinite $n \times n$ matrices with eigenvalues belonging to $\left[ \lambda_{-} , \lambda_{+} \right]$. The 
 logarithm of a matrix $X \in \R^{n \times n}$ is the matrix $Y := \log \left( X \right) \in  \R^{n \times n}$ such that
\begin{equation*}
X = e^{Y} = \sum_{i \geq 0} \dfrac{1}{i!} Y^{i} .
\end{equation*}

The functions used in the numerical experiments were:
\begin{enumerate}[label=$(\alph*)$]
\item 
``\emph{$\log\det$ $+$ Nuclear norm}''
 \begin{equation*}
f: \R^{n \times n} \rightarrow \R \cup \left\lbrace + \infty \right\rbrace, \quad f \left( X \right) := \begin{cases}
- \log \left( \det \left( X \right) \right) + \mu \norm{X}_{\ast} & \textrm{ if } X \succ 0 , \nonumber \\
+ \infty & \textrm{ otherwise},
\end{cases}
\end{equation*}

\item 
``\emph{$\log\det$ $+$ Squared Frobenius norm}''
\begin{equation*}
f: \R^{n \times n} \rightarrow \R \cup \left\lbrace + \infty \right\rbrace, \quad f \left( X \right) := \begin{cases}
- \log \left( \det \left( X \right) \right) + \mu \norm{X}_{\mathsf{F}}^{2} & \textrm{ if } X \succ 0 , \nonumber \\
+ \infty & \textrm{ otherwise},
\end{cases}
\end{equation*}

 \item 
 ``\emph{$\log\det$ $+$ Bounds on eigenvalues}''
 \begin{equation*}
f: \R^{n \times n} \rightarrow \R \cup \left\lbrace + \infty \right\rbrace, \quad f \left( X \right) := \begin{cases}
 - \log \left( \det \left( X \right) \right) & \textrm{ if } X \in E_{\lambda_{-},\lambda_{+}} \left( \R^{n \times n} \right) , \nonumber \\
 + \infty & \textrm{ otherwise},
 \end{cases}
 \end{equation*}

\item 
``\emph{von Neumann entropy	$+$ Nuclear norm}''
\begin{equation*}
f: \R^{n \times n} \rightarrow \R \cup \left\lbrace + \infty \right\rbrace, \quad f \left( X \right) := \begin{cases}
\mathrm{trace} \left( X \log \left( X \right) \right) + \mu \norm{X}_{\ast} & \textrm{ if } X \succeq 0 , \nonumber \\
+ \infty & \textrm{ otherwise},
\end{cases}
\end{equation*}

\item 
``\emph{von Neumann entropy	$+$ Squared Frobenius norm}''
\begin{equation*}
f: \R^{n \times n} \rightarrow \R \cup \left\lbrace + \infty \right\rbrace, \quad f \left( X \right) := \begin{cases}
\mathrm{trace} \left( X \log \left( X \right) \right) + \mu \norm{X}_{\mathsf{F}}^{2} & \textrm{ if } X \succeq 0 , \nonumber \\
+ \infty & \textrm{ otherwise},
\end{cases}
\end{equation*}

 \item 
 ``\emph{von Neumann entropy	$+$ Bounds on eigenvalues}''
 \begin{equation*}
f: \R^{n \times n} \rightarrow \R \cup \left\lbrace + \infty \right\rbrace, \quad f \left( X \right) := \begin{cases}
 \mathrm{trace} \left( X \log \left( X \right) \right) & \textrm{ if } X \in E_{\lambda_{-},\lambda_{+}} \left( \R^{n \times n} \right) , \nonumber \\
 + \infty & \textrm{ otherwise},
 \end{cases}
 \end{equation*}

\item 
``\emph{Ridge $+$ Schatten $q$-penalty}''
\begin{equation*}
f: \R^{m \times n} \rightarrow \R,  \quad f \left( X \right) := \dfrac{1}{2} \norm{X}_{\mathsf{F}}^{2} + \mu \norm{X}_{\mathcal{S}_{q}}^{q} .
\end{equation*}
\end{enumerate}

For all cases we considered $\mu := 10^{-1}$ and solved the optimization problem \eqref{pb:Moreau} for $\gamma:=1$.  For this purpose we used Algorithm \ref{algo:prox-vbased} and Algorithm \ref{algo:prox-gbased} when $p:=2$, the classical proximal point algorithm (PPA) (\cite{BaCo,Beck}) and the special case of FISTA obtained in the absence of the smooth term (\cite{BT}),  which we also compared with each other. For the formulas of the proximal operators of the seven objective functions, see \cite{Beck,BCP,CCCP}.

We have set $n := 10^2$, and for the Ridge $+$ Schatten $q$-penalty we have chosen $m \in \left\lbrace 1 , 10^{2} , 10^{4} \right\rbrace$ to see to what extent the dimension of the matrix affects the numerical performance, also for different values of $q$.  The initial point $Y_{0}$ has been generated randomly, whereas for Algorithm \ref{algo:prox-vbased} we have simply added $1$ to all its entries to obtain $Y_{-1}$.

We have stopped the algorithms either when
\begin{equation*}
\dfrac{\norm{Y_{k+1}-Y_{k}}}{\norm{Y_{k+1}}} < \mathtt{Tol}
\end{equation*}
or if they a maximum allowed number of iterations $\mathtt{Ite\_max}$ has been exceeded.  We have set $\mathtt{Tol}:=10^{-16}$ and $\mathtt{Ite\_max} := 10^{3}$.  The performance of the four algorithms has been compared in terms of $f \left( \prox_{f} \left( Y_{k} \right) \right) - f_{*}$ in logarithmic scale,  where $f_*$ is the minimum value the objective function takes over all generated sequences.

According to the main convergence theorems,   for $p=2$,  Algorithm \ref{algo:prox-vbased} and Algorithm \ref{algo:prox-gbased} have rates of convergence of  $o \left( k^{-3/2} \right)$ as $k \rightarrow +\infty$,  which is faster than that of PPA, but worse than that of FISTA.  As Figure \ref{fig:PPA-FISTA-PEAS-sym} shows for the objective functions (a) - (f),  Algorithm \ref{algo:prox-vbased} and Algorithm \ref{algo:prox-gbased} outperform PPA and FISTA on almost all instances considered.  On the other hand, they do not exhibit oscillations, a phenomenon known to occur with momentum algorithms. 
The better convergence properties of the algorithms presented in this paper are confirmed for the different instances considered in th Figures \ref{fig:PPA-FISTA-PEAS-ridge-1e0} - \ref{fig:PPA-FISTA-PEAS-ridge-1e4} when minimizing the \emph{Ridge $+$ Schatten $q$-penalty function},  especially as the values for $q$ become larger.  We believe that the consistent improvements in the convergence of the function values observed for  Algorithm \ref{algo:prox-vbased} and Algorithm \ref{algo:prox-gbased} are due to the fact that these algorithms make much better use of local information and adjust the step size accordingly.

\graphicspath{ {plots/} }

\begin{figure}[!htb]
\centering
\begin{subfigure}{0.48\textwidth}
\includegraphics[width=\linewidth]{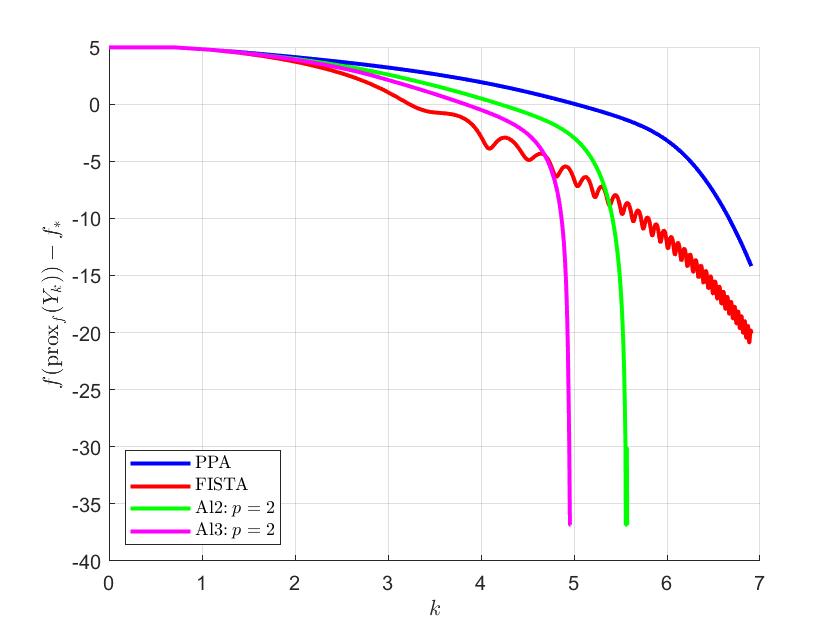}
\caption{``\emph{$\log\det$ $+$ Nuclear norm}''}
\par\medskip
\includegraphics[width=\linewidth]{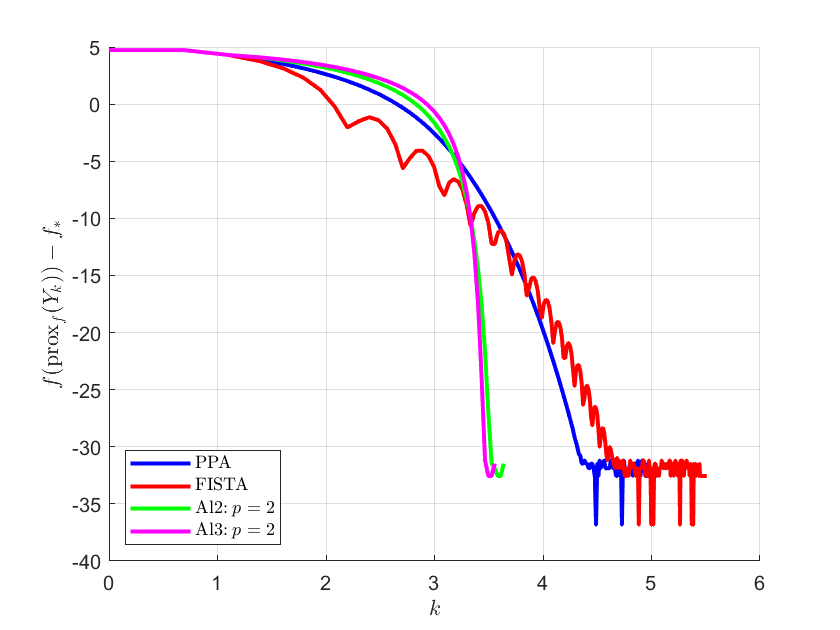}
\caption{``\emph{$\log\det$ $+$ Squared Frobenius norm}''}
\par\medskip
\includegraphics[width=\linewidth]{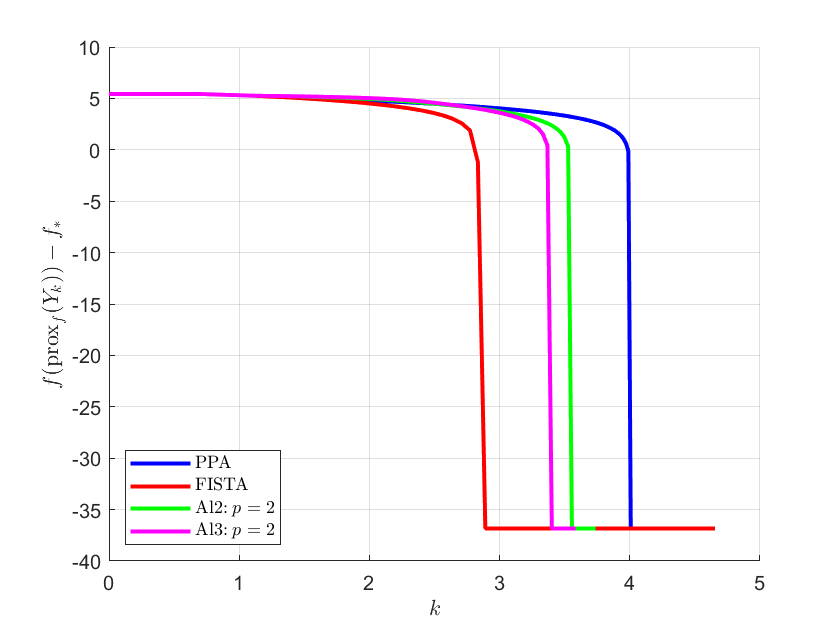}
\caption{``\emph{$\log\det$ $+$ Bounds on eigenvalues}''}
\end{subfigure}\hfill
\begin{subfigure}{0.48\textwidth}
\includegraphics[width=\linewidth]{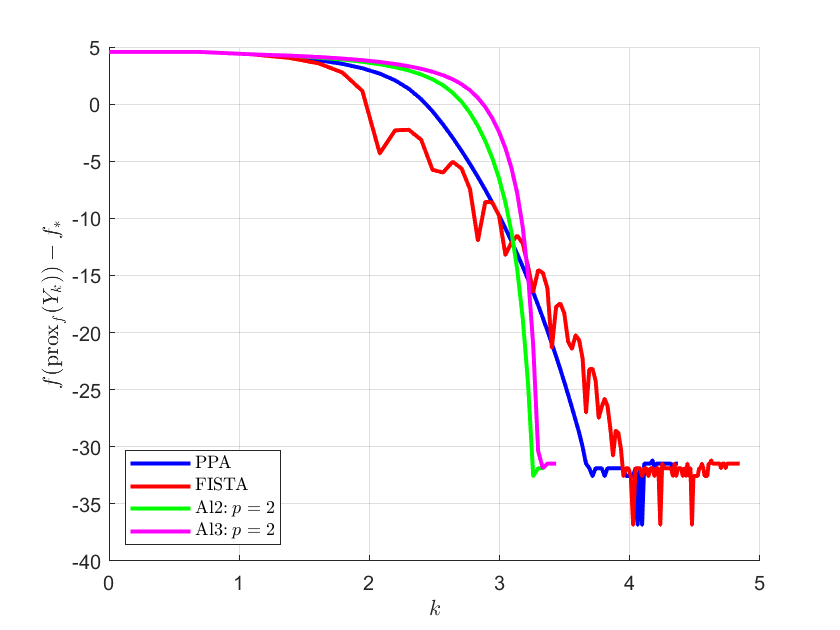}
\caption{``\emph{von Neumann entropy $+$ Nuclear norm}''}
\par\medskip
\includegraphics[width=\linewidth]{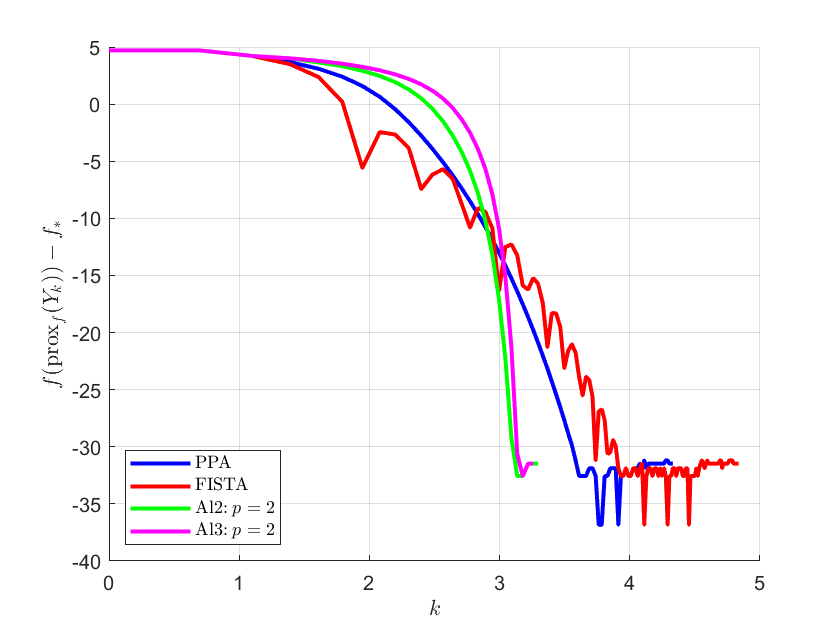}
\caption{``\emph{von Neumann entropy $+$ Squared Frobenius norm}''}
\par\medskip
\includegraphics[width=\linewidth]{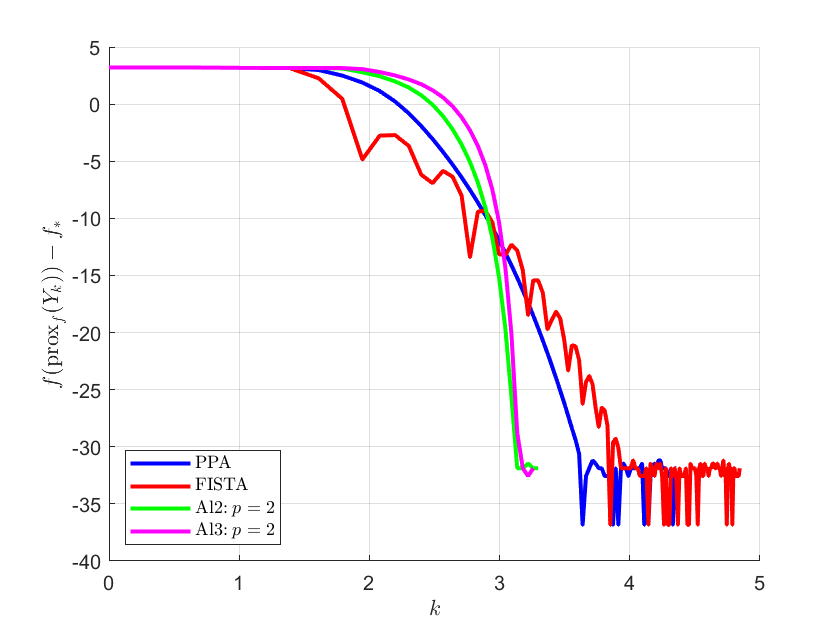}
\caption{``\emph{von Neumann entropy $+$ Bounds on eigenvalues}''}
\end{subfigure}
\caption{Numerical comparisons between PPA, FISTA, and Algorithm \ref{algo:prox-vbased} and Algorithm \ref{algo:prox-gbased}.}
\label{fig:PPA-FISTA-PEAS-sym}
\end{figure}

\begin{figure}[!htb]
\centering
\begin{subfigure}{0.48\textwidth}
\includegraphics[width=\linewidth]{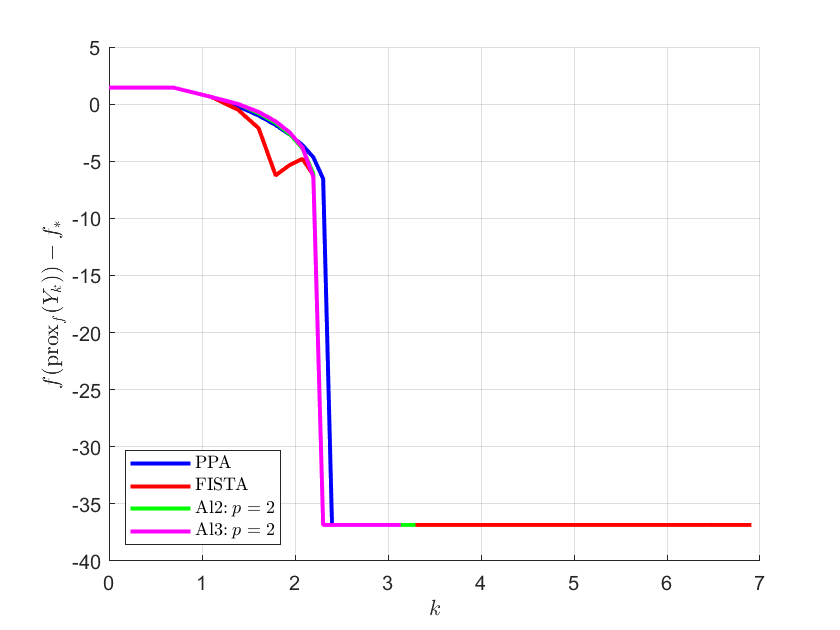}
\caption{``\emph{Ridge $+$ Nuclear norm}''}
\par\medskip
\includegraphics[width=\linewidth]{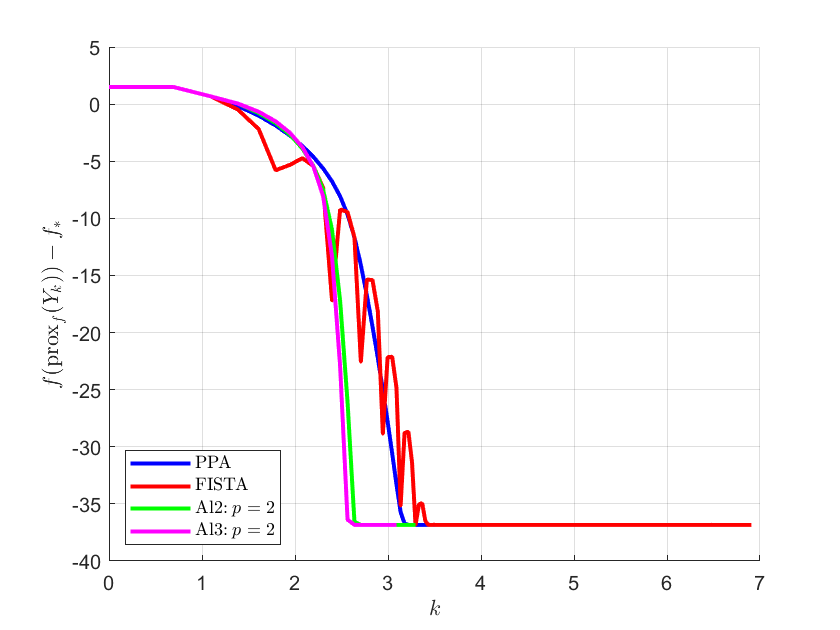}
\caption{``\emph{Ridge $+$ Schatten $4/3$-penalty}''}
\end{subfigure}\hfill
\begin{subfigure}{0.48\textwidth}
\includegraphics[width=\linewidth]{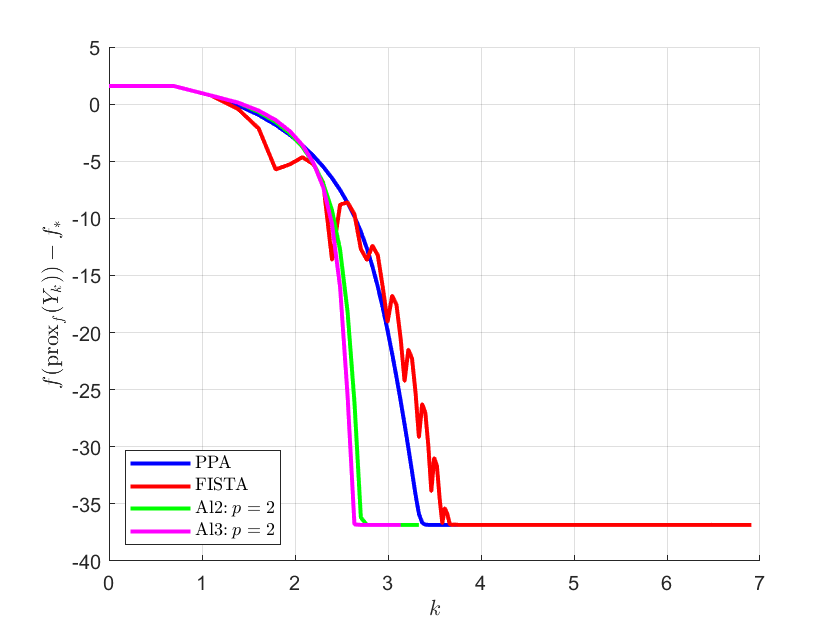}
\caption{``\emph{Ridge $+$ Schatten $3/2$-penalty}''}
\par\medskip
\includegraphics[width=\linewidth]{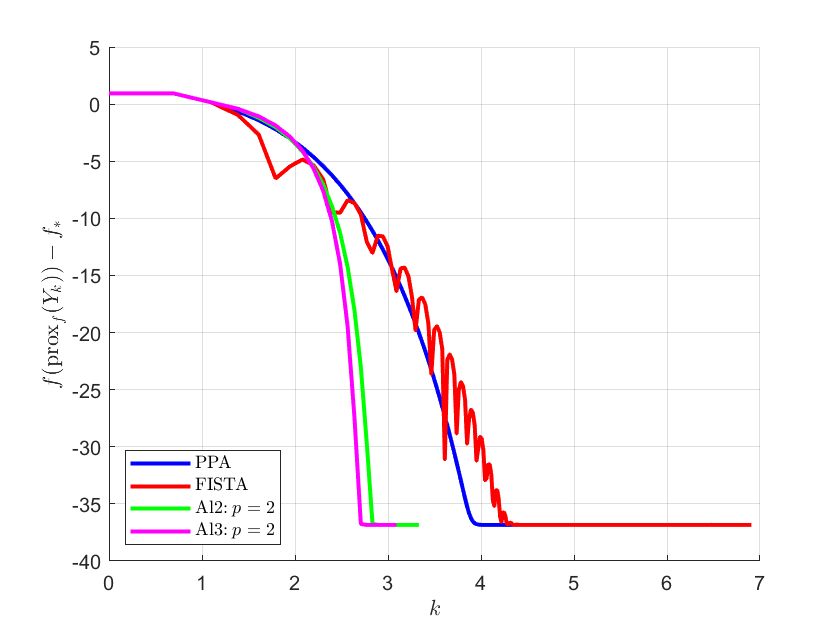}
\caption{``\emph{Ridge $+$ Schatten $4$-penalty}''}
\end{subfigure}
\caption{Numerical comparisons between PPA,  FISTA, and Algorithm \ref{algo:prox-vbased} and Algorithm \ref{algo:prox-gbased} for $\left( m , n \right) := \left( 1 , 10^{2} \right)$.}
\label{fig:PPA-FISTA-PEAS-ridge-1e0}
\end{figure}

\begin{figure}[!htb]
\centering
\begin{subfigure}{0.48\textwidth}
\includegraphics[width=\linewidth]{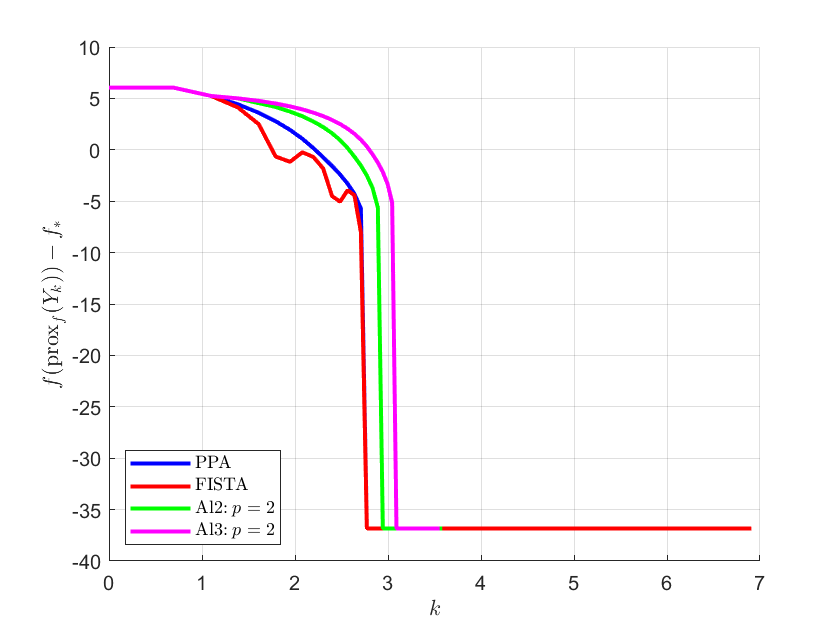}
\caption{``\emph{Ridge $+$ Nuclear norm}''}
\par\medskip
\includegraphics[width=\linewidth]{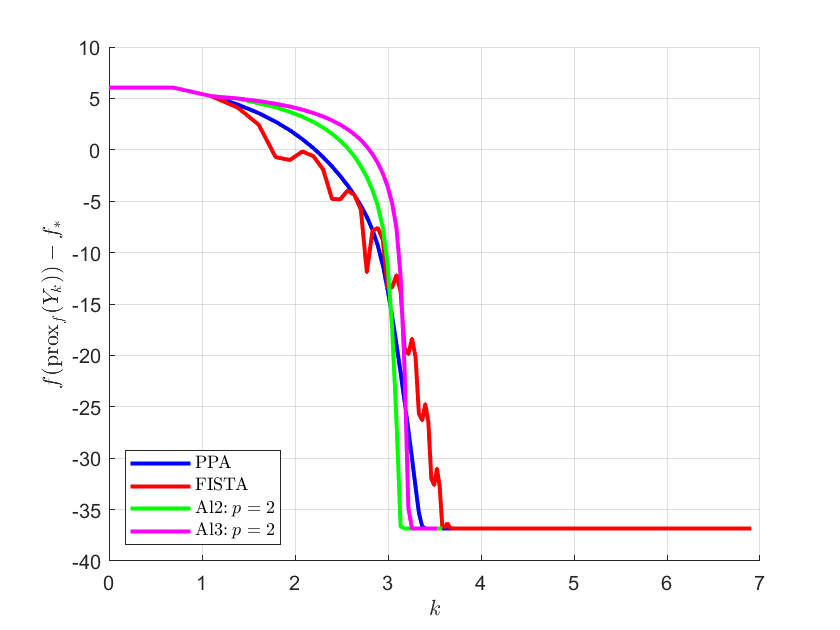}
\caption{``\emph{Ridge $+$ Schatten $4/3$-penalty}''}
\end{subfigure}\hfill
\begin{subfigure}{0.48\textwidth}
\includegraphics[width=\linewidth]{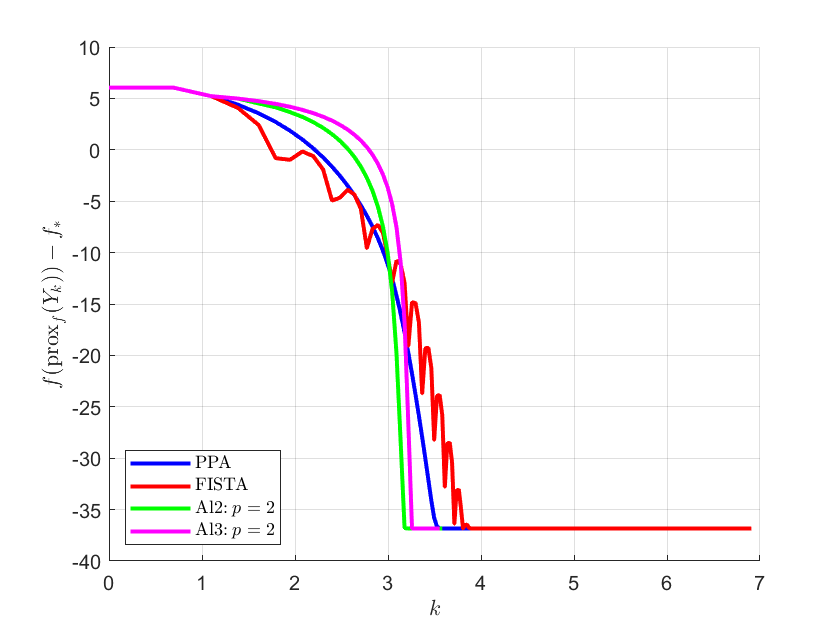}
\caption{``\emph{Ridge $+$ Schatten $3/2$-penalty}''}
\par\medskip
\includegraphics[width=\linewidth]{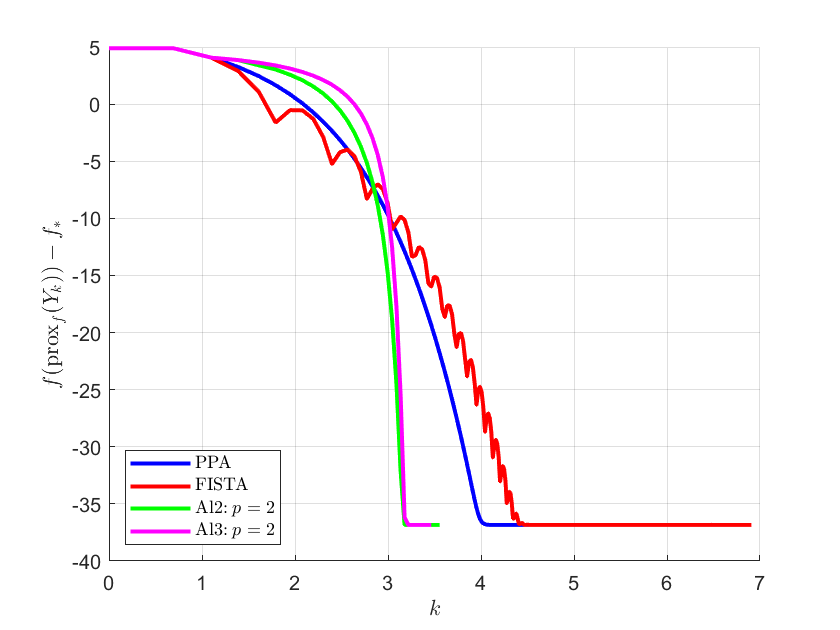}
\caption{``\emph{Ridge $+$ Schatten $4$-penalty}''}
\end{subfigure}
\caption{Numerical comparisons between PPA,  FISTA, and Algorithm \ref{algo:prox-vbased} and Algorithm \ref{algo:prox-gbased} for  $\left( m , n \right) := \left( 10^{2} , 10^{2} \right)$.}
\label{fig:PPA-FISTA-PEAS-ridge-1e2}
\end{figure}

\begin{figure}[!htb]
\centering
\begin{subfigure}{0.48\textwidth}
\includegraphics[width=\linewidth]{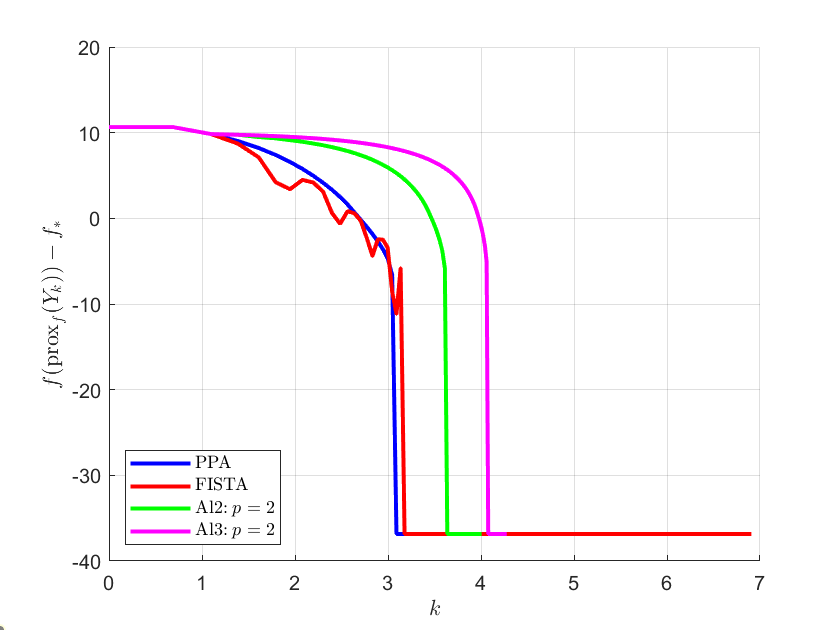}
\caption{``\emph{Ridge $+$ Nuclear norm}''}
\par\medskip
\includegraphics[width=\linewidth]{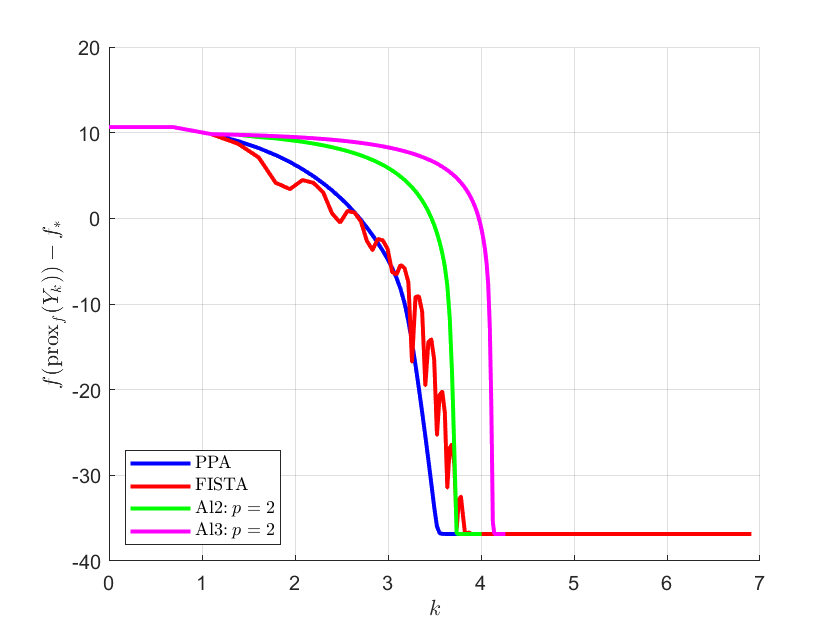}
\caption{``\emph{Ridge $+$ Schatten $4/3$-penalty}''}
\end{subfigure}\hfill
\begin{subfigure}{0.48\textwidth}
\includegraphics[width=\linewidth]{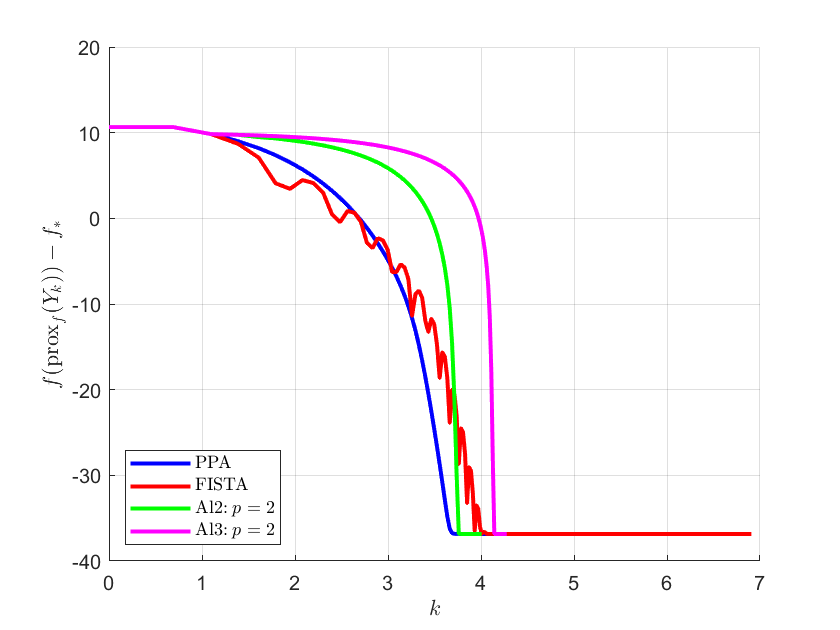}
\caption{``\emph{Ridge $+$ Schatten $3/2$-penalty}''}
\par\medskip
\includegraphics[width=\linewidth]{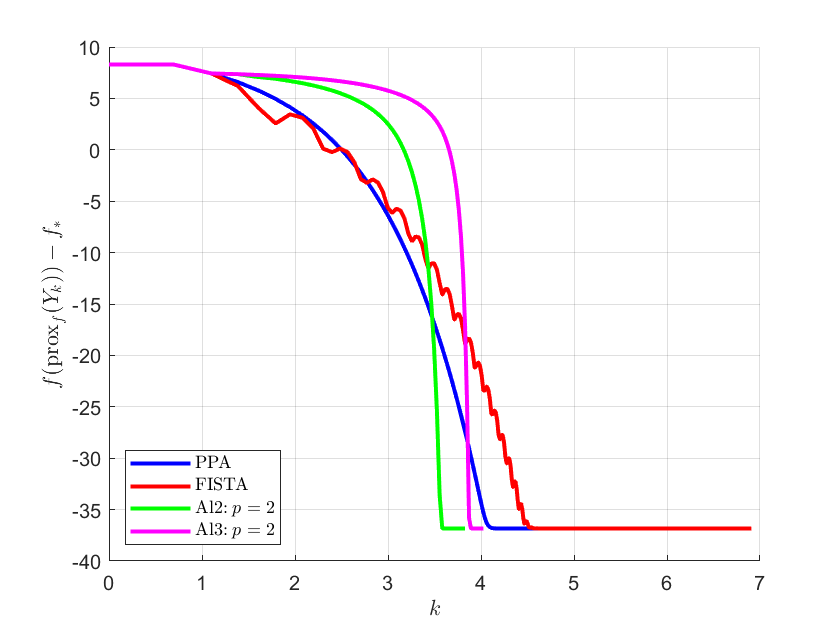}
\caption{``\emph{Ridge $+$ Schatten $4$-penalty}''}
\end{subfigure}
\caption{Numerical comparisons between PPA,  FISTA, and Algorithm \ref{algo:prox-vbased} and Algorithm \ref{algo:prox-gbased} for  $\left( m , n \right) := \left( 10^{4} , 10^{2} \right)$.}
\label{fig:PPA-FISTA-PEAS-ridge-1e4}
\end{figure}

\section{Conclusion and perspective}

Our study proposes new fast adaptive optimization methods for convex optimization. We have shown that the time scaling and averaging technique, previously developed by the authors in the context of non-autonomous systems, can be developed by taking closed-loop time parameterization, giving rise to autonomous dynamics.
The method turns out to be flexible, because it is based on elementary mathematical tools, namely the dynamics of the steepest descent, and the operations of temporal parameterization and averaging. It is therefore not necessary to redo a Lyapunov analysis, one relies on the classic results for the steepest descent.
The results obtained for the continuous dynamics pass quite naturally to the corresponding proximal algorithms, where the iterates are expressed in a direct way according to the proximal terms.
This study is one of the very first to develop an algorithmic framework based on autonomous dynamics and which, when specialized, provides the convergence rates of the dynamical surrogate of the  Nesterov acceleration gradient method.
Another important aspect of our analysis is that it  exhibits Hessian-driven damping, which plays a key role in damping oscillations.
Our work opens up many perspectives, our method naturally extending to gradient algorithms, proximal-gradient algorithms for composite optimization, cocercive monotone operators, and the study of the stochastic version, to name only a few.

\section{Appendix}

\subsection{Classical facts concerning the continuous steepest descent}

Consider the classical continuous steepest descent

\begin{equation}\label{SD_appendix}
{\rm (SD)} \quad   \dot{z}(t) + \nabla f (z(t)) =0.
\end{equation}
 Under the standing assumption $(\mathcal A)$ on $f$,  we know that, for any $z_0 \in \cH$ there exists a unique classical global solution 
$z\in \mathcal C^1([t_0, +\infty[: \cH )$   of (SD) satisfying  $z(t_0 )= z_0$, see \cite[Theorem 17.1.1]{ABM_book}. We fix $t_0 $ as the origin of time.
Recall classical facts concerning the continuous steepest descent.

\medskip

\begin{theorem}\label{SD_pert_thm}
Suppose that  $f \colon \cH \to \R$ satisfies $(\mathcal A)$.
Let $z \colon \left[ t_{0} , + \infty \right[ \to \cH$ be a solution trajectory of 

\begin{equation}\label{pert SD}	
\dot{z}(t) + \nabla f( z(t)) = g(t)
\end{equation}
where  $g \colon \left[ t_{0} , + \infty \right[ \to \cH$ is such that
\begin{equation}\label{pert 01}
	\int_{t_{0}}^{+ \infty} \left\lVert g \left( t \right) \right\rVert dt < + \infty  \textrm{ and } \int_{t_{0}}^{+ \infty} t \left\lVert g \left( t \right) \right\rVert ^{2} dt < + \infty .
\end{equation}
Then the following statements are satisfied:

\begin{enumerate}
\item (convergence  of  gradients towards zero)	\quad $\left\lVert \nabla f \left( z \left( t \right) \right) \right\rVert = o \left( \dfrac{1}{\sqrt{t}} \right) \textrm{ as } t \to + \infty.
$

\item (integral estimate of the velocities)
$
\displaystyle{\int_{t_0}^{+\infty} } t \left\lVert \dot{z} \left( t \right) \right\rVert ^{2} dt < +\infty.
$

\item (integral estimate of the gradients)
$
\displaystyle{\int_{t_0}^{+\infty} } t \left\lVert \nabla f \left( z \left( t \right) \right) \right\rVert ^{2} dt < +\infty.
$

\item (convergence of values)
$
f \left( z \left( t \right) \right) -\inf\nolimits_{\cH} f = o \left( \dfrac{1}{t} \right) \textrm{ as } t \to + \infty .
$

\item (improved convergence rates of gradients)
if $g(t) \equiv 0$, then
$
\left\lVert \nabla f \left( z \left( t \right) \right) \right\rVert = o \left( \dfrac{1}{t} \right) \textrm{ as } t \to + \infty .
$

\item 
The solution trajectory $z(t)$ converges weakly as $t \to +\infty$, and its limit belongs to $\cS=\argmin f$.
\end{enumerate}
\end{theorem}

If $g(t) \equiv 0$ we have that $t \mapsto \left\lVert \nabla f \left( z \left( t \right) \right) \right\rVert$ is nonincreasing since in this case $$\frac{d}{dt} \left\lVert \nabla f \left( z \left( t \right) \right) \right\rVert ^{2} = 2 \left\langle \nabla f \left( z \left( t \right) \right) , \frac{d}{dt} \nabla f \left( z \left( t \right) \right) \right\rangle = - 2 \left\langle \dot{z} \left( t \right) , \frac{d}{dt} \nabla f \left( z \left( t \right) \right) \right\rangle \leq 0 \quad \forall t \geq t_0.$$
Therefore, from the integral estimate of the gradients we deduce that $\left\lVert \nabla f \left( z \left( t \right) \right) \right\rVert = o \left( \frac{1}{t} \right)$.

\subsection{Auxiliary result}

Opial's Lemma is a basic ingredient of the convergence analysis.
\begin{lemma}\label{Opial} (Opial) Let $S$ be a nonempty subset of $\cH$ and let $\left( x_{k} \right) _{k \geq 0}$ be a sequence in $\cH$. Assume that 
	\begin{enumerate}
		\item for every $z \in S$, $\lim_{k \to + \infty} \left\lVert x_{k} - z \right\rVert$ exists;
		\item every weak sequential limit point of $\left( x_{k} \right) _{k \geq 0}$, as $k \to + \infty$, belongs to $S$.
	\end{enumerate}
	Then  $\left( x_{k} \right) _{k \geq 0}$ converges weakly as $k \to +\infty$, and its limit belongs to $S$.
\end{lemma}

\section*{Funding}

The research of RIB and DKN has been supported by FWF (Austrian Science Fund), projects W 1260 and P 34922-N, respectively.

\section*{Acknowledgements}

This work was completed in the final year of Hedy Attouch’s life,  just months before his passing.  Radu Ioan Bo\c t and
Dang-Khoa Nguyen wish to take this opportunity to pay tribute to a remarkable mathematician, a kind and generous
soul, whose absence is deeply felt.


\end{document}